\documentclass[10pt,A4]{article}
\usepackage{amsmath,amsfonts,amssymb,amsthm}
\usepackage{mathrsfs}
\usepackage{latexsym}
\usepackage[english]{babel}
\def\vt{\vartheta}
\def\Cov{{\rm Cov\,}}

\newcommand{\dlines}{\displaylines}
\newcommand{\field}[1]{\mathbb{#1}}
\newcommand{\R}{\field{R}}

\renewcommand{\th}{\theta}

\newcommand{\X}{\mathscr{X}}

\newcommand{\Z}{\field{Z}}

\newcommand{\bS}{\field{S}}

\newcommand{\C}{\field{C}}

\newcommand{\Var}{{\rm Var}}

\newcommand{\F}{{\mathscr{F}}}

\newcommand{\B}{{\mathscr B}}

\newcommand{\cS}{{\mathcal S}}
\newcommand{\goto}{{\longrightarrow}}

\def\tin#1{\par\noindent\hskip3em\llap{#1\enspace}\ignorespaces}

\def\cS{{\field S}}
\def\E{{\mathbb{ E}}}

\def\P{{\mathbb{P}}}
\def\F{{\mathscr{F}}}

\def\mE{\mathscr E}
\def\paref#1{(\ref{#1})}
\def\tfrac#1#2{{\textstyle\frac {#1}{#2}}}
\newtheorem{theorem}{Theorem}[section]
\newtheorem{rema}[theorem]{Remark}
\numberwithin{equation}{section}

\newtheorem{cor}[theorem]{Corollary}
\newtheorem{definition}[theorem]{Definition}
\newtheorem{example}[theorem]{Example}

\newtheorem{lemma}[theorem]{Lemma}
\newtheorem{prop}[theorem]{Proposition}
\newtheorem{remark}[theorem]{Remark}

\def\tr{\mathop{\rm tr}}
\begin{document}

\title{\Huge\sc Representation of Gaussian Isotropic Spin Random Fields}

\author{Paolo Baldi, Maurizia Rossi\footnote{Research Supported by ERC Grant 277742 \it{Pascal}}\\
{\it Dipartimento di Matematica, Universit\`a
di Roma Tor Vergata, Italy}\\
} \date{}\maketitle

\begin{abstract}
We develop a technique for the construction of random fields on
algebraic structures.
We deal with two general situations: random fields on homogeneous
spaces of a compact group and in the spin line bundles of the
$2$-sphere. In particular, every complex Gaussian isotropic spin random field can be represented in this way.
Our construction extends P. L\'evy's original idea for the spherical Brownian motion.
\end{abstract}

\noindent{\it AMS 2000 subject classification:} 60G60, 33C55, 57T30
\smallskip

\noindent{\it Key words and phrases:} Gaussian random fields, spherical harmonics, line bundles, induced representations

\section{Introduction}

In recent years the investigation of random fields on algebraic structures has received a renewed
interest from a theoretical point of view (\cite{PePy}), but mainly motivated by applications concerning the
modeling of the Cosmic Microwave Background data (see \cite{BM06},  \cite{MR2578835}, \cite{MA:2005}
 e.g. and also the book \cite{dogiocam}). Actually one of the features of this radiation,
the {\it temperature}, is well modeled as a single realization of a random field on the sphere $\cS^2$.

Moreover thanks to the ESA satellite Planck mission, new data concerning the {\it polarization}
of the CMB will soon be available and the modeling of this quantity has led quite naturally to the
investigation of spin random fields on $\cS^2$, a
subject that has already received much attention (see \cite{marinuccigeller},\cite{bib:LS}, \cite{bib:M} and again \cite{dogiocam} chapter 12 e.g.).

The object of this paper is the investigation of Gaussian isotropic
 random fields or, more precisely random sections, in the homogeneous line
bundles of $\cS^2$. In this direction we first investigate the simpler situation
of random fields on the homogeneous space $\X=G/K$ of a compact group $G$. Starting from
P. L\'evy's construction of his ``mouvement brownien fonction d'un point de la sph\`ere de
{R}iemann'' in \cite{bib:L}, we prove that to every square integrable bi-$K$-invariant
function $f:G\to\R$ a real Gaussian isotropic random field on $\X$ can be associated
 and also that every real Gaussian
isotropic  random field on $\X$ can be obtained in this way.

Then, turning to our main object, we prove first that every random section of a homogeneous
vector bundle on $\X$ is actually ``equivalent'' to a random field on the group $G$,
enjoying a specific pathwise invariance property. This fact is of great importance as it reduces the
investigation of random sections of the homogeneous vector bundle to a much simpler situation and is our key tool for the results that follow.

We are then able to give a method for constructing Gaussian isotropic random sections of the
homogeneous line bundles of $\cS^2$ and prove that every complex Gaussian isotropic  random
section can actually be obtained in this way. More precisely given $s\in \Z$,
we prove that to every function $f:SO(3)\to\C$ which is bi-$s$-associated, i.e. that
transforms under both the right and left action of
the isotropy group $K\simeq SO(2)$ according to its $s$-th linear character,
one can associate a Gaussian isotropic random section
of the $s$-homogeneous line bundle and also that every complex Gaussian
isotropic random section is
represented in this way. In some sense this extends the representation result for Gaussian isotropic  random
fields on homogeneous spaces described above: a bi-$K$-invariant function
being associated to the $0$-th
character and a random field on $\cS^2$ being also a
random section of the $0$-homogeneous line bundle.

In \cite{bib:L} P. L\'evy proves the existence on the spheres $\cS^m$, $m\ge 1$,
of a Gaussian random field $T$ such that $T_x-T_y$ is normally distributed
with variance $d(x,y)$, $d$ denoting the distance on $\cS^m$.
The existence of such a random field on a more general Riemannian manifold
has been the object of a certain number of papers
(\cite{mauSO3}, \cite{bib:G}, \cite{bib:TKU} e.g.).
We have given some thought about what a P. L\'evy random section should be.
We think that our treatment should be very useful in this direction but we
have to leave this as an open question.

The plan is as follows. In \S2 we recall general facts concerning
Fourier expansions on the homogeneous space $\X$ of a compact group $G$. In \S3 we introduce the topic
of isotropic random fields on $\X$ and their relationship with positive definite functions enjoying certain invariance properties. These facts are the basis for the
representation results for real Gaussian isotropic random fields
on $\X$ which are obtained in \S4.

In \S5 we introduce the subject of random sections of homogeneous vector bundles on $\X$ and in particular in \S6 we consider the case $\X=\cS^2$. We were much inspired here by the approach of \cite{bib:M}, but the introduction of  pullback
random fields considerably simplifies the understanding of the notions that are involved.

In \S7 we extend  the construction of \S4 as described above and prove the
representation result for complex Gaussian isotropic random sections of
the homogeneous line bundles of $\cS^2$.
When dealing with such random sections different approaches are available in the literature.
They are equivalent but we did not find a formal proof of this equivalence. 
So \S8 is devoted to the comparison with other constructions 
(\cite{marinuccigeller}, \cite{bib:LS}, \cite{bib:M},
\cite{bib:NP}). In particular we show that the construction of spin random fields on $\cS^2$ of \cite{marinuccigeller} is actually
equivalent to ours, the difference being that their point of view is based on an accurate description of local charts and transition functions, instead of our more global perspective.
\section{Fourier expansions}
Throughout this paper $\X$ denotes the (compact) homogeneous space
of a topological \emph{compact} group $G$. Therefore $G$ acts
transitively on $\X$ with an action $x\mapsto gx, g\in G$. $\B(\X)$,
$\B(G)$ stand for the Borel $\sigma$-fields of $\X$ and $G$
respectively. We shall denote $dg$ the Haar measure of $G$ and $dx$
the corresponding $G$-invariant measure on $\X$ and assume, unless explicitly stated, that both these measures have total mass equal to $1$. Actually at some places (mainly in \S6, \S7 and \S8) we consider for $\cS^2$ the total mass $4\pi$ in order to
be consistent with the existing literature.
We shall write $L^2(G)$ for $L^2(G, dg)$ and similarly $L^2(\X)$
for $L^2(\X, dx)$. Unless otherwise stated the $L^2$-spaces are
spaces of \emph{complex valued} square integrable functions.

We fix once forever a point $x_0\in\X$ and denote $K$ the
isotropy group of $x_0$, i.e. the subgroup of $G$ of the elements $g$
such that $gx_0=x_0$, so that $\X\cong G/K$. In the case
$G=SO(3)$, $\X=\cS^2$, $x_0$ will be  the north pole, as usual, and $K\cong SO(2)$.

Given a function $f:\X\to\C$ we define its \emph{pullback} as
\begin{equation}
\widetilde f(g) := f(gx_0)
\end{equation}
which is a $G\to\C$ right-$K$-invariant function, i.e.
constant on left cosets of $K$. By the
integration rule of image measures it holds
\begin{equation}\label{int-rule}
\int_\X f(x)\, dx=\int_G \widetilde f(g)\, dg\ ,
\end{equation}
whenever one of the integrals has sense.

We denote $L$ the \emph{left regular representation} of $G$ on
$L^2(G)$ given by $L_g f(h):= f(g^{-1}h)$ for  $g\in G$ and $f\in
L^2(G)$, and $\widehat G$ the \emph{dual} of $G$, i.e., the set of
the equivalence classes of irreducible unitary representations of
$G$. As $G$ is assumed to be compact,  $\widehat G$ is at most
countable.

For every $\sigma\in \widehat G$ let $(D^\sigma, H_\sigma)$ be a representative element
where the (unitary) operator $D^\sigma$ acts
irreducibly on the complex (finite dimensional) Hilbert space
$H_\sigma$. We denote $\langle \cdot, \cdot \rangle$ the inner
product of $H_\sigma$ and $\dim\sigma=\dim H_\sigma$.

Given $f\in
L^2(G)$, for every $\sigma\in \widehat G$ we define
\begin{equation}\label{Fourier coefficient}
\widehat f(\sigma) :=\sqrt{\dim \sigma}\int_G f(g)
D^\sigma(g^{-1})\,dg
\end{equation}
which is a linear endomorphism of $H_\sigma$. It is immediate that, denoting $*$ the convolution on $G$,
\begin{equation}\label{conv1}
\widehat{f_1\ast f_2}(\sigma)=\frac 1{\sqrt{\dim \sigma}}\,\widehat
f_2(\sigma) \widehat f_1(\sigma)\ .
\end{equation}
The Peter-Weyl Theorem (see \cite{dogiocam} or \cite{sugiura} e.g.)
can be stated as
\begin{equation}\label{PW for compact groups}
f(g) = \sum_{\sigma\in \widehat G} \sqrt{\dim \sigma} \tr\bigl(\widehat
f(\sigma) D^\sigma(g)\bigr)\ ,
\end{equation}
the convergence of the series taking place in $L^2(G)$.
Fix any orthonormal basis $v_1, v_2, \dots, v_{\dim\sigma }$ of
$H_\sigma$ and denote $D^\sigma_{ij}(g):= \langle D^\sigma(g) v_j, v_i\rangle$
the $(i,j)$-th coefficient of
the matrix representation of $D^\sigma(g)$ with respect to this
basis, then the matrix representation of $\widehat f(\sigma)$ has entries
$$
\dlines{\widehat f(\sigma)_{i,j}=\langle \widehat f(\sigma) v_j,
v_i\rangle = \sqrt{\dim \sigma} \int_G f(g) \langle D^\sigma(g^{-1})
v_j, v_i\rangle \,dg = \cr = \sqrt{\dim \sigma} \int_G f(g)
D^\sigma_{i,j}(g^{-1})\,dg = \sqrt{\dim \sigma} \int_G f(g)
\overline{D^\sigma_{j,i}(g)}\,dg \cr }
$$
and the Peter-Weyl Theorem (\ref{PW for compact groups}) becomes
\begin{equation}\label{PW1}
f(g)=\sum_{\sigma\in \widehat G} \sqrt{\dim \sigma} \sum_{i,j=1}^{\dim\sigma
} \widehat f(\sigma)_{j,i} D^\sigma_{i,j}(g)\ ,
\end{equation}
the above series converging in $L^2(G)$.
Let $L^2_\sigma(G)\subset L^2(G)$ be the isotypical subspace of $\sigma\in\widehat G$,
 i.e. the subspace generated by the functions $D^\sigma_{i,j}$; it is a $G$-module that can be decomposed into the
orthogonal direct sum of $\dim\sigma $ irreducible and
equivalent $G$-modules
$(L^2_{\sigma,j}(G))_{j=1,\dots,\dim\sigma }$ where each
$L^2_{\sigma,j}(G)$ is spanned by the functions
$D^\sigma_{i,j}$ for $i=1,\dots, \dim\sigma $, loosely
speaking by the $j$-th column of the matrix $D^\sigma$.

If we denote
\begin{equation}\label{component}
f^\sigma(g)=\sqrt{\dim \sigma} \tr(\widehat f(\sigma) D^\sigma(g))
\end{equation}
the component (i.e. the projection) of $f$ in $L^2_\sigma(G)$, then \paref{PW for compact groups} and \paref{PW1} become
\begin{equation}\label{PWdiverso}
f(g)=\sum_{\sigma\in\widehat G}f^\sigma(g)
\end{equation}
 or equivalently the Peter-Weyl Theorem can be stated as
\begin{equation}\label{PW2}
L^2(G)= \bigoplus_{\sigma \in \widehat G} \,\, \mathop{\oplus}\limits_{j=1}^{\dim\sigma }
L^2_{\sigma,j}(G)\ ,
\end{equation}
the direct sums being orthogonal.

The Fourier expansion of functions $f\in L^2(\X)$ can be obtained
easily just remarking that, thanks to \paref{int-rule}, their pullbacks $\widetilde f$
belong to $L^2(G)$ and form a $G$-invariant closed subspace of $L^2(G)$.
We can therefore associate to $f\in L^2(\X)$ the family of operators $\bigl(
\widehat {\widetilde f}(\sigma) \bigr)_{\sigma\in \widehat G}$.
Let $H_{\sigma,0}$ denote the subspace of $H_\sigma$ (possibly
reduced to $\{0\}$) formed by the vectors that remain fixed under the action of
$K$, i.e. for every $k\in K, v\in
H_{\sigma,0}$ $D^\sigma(k)v=v$.
Right-$K$-invariance implies that the image of $\widehat
{\widetilde f}(\sigma)$ is contained in $H_{\sigma,0}$:
\begin{equation}\label{proiezione triviale}
\begin{array}{c}
\displaystyle\widehat {\widetilde f}(\sigma) =\sqrt{\dim \sigma}
\int_G \widetilde f(g) D^\sigma(g^{-1})\,dg=\sqrt{\dim \sigma}
\int_G \widetilde f(gk) D^\sigma(g^{-1})\,dg=\\
\displaystyle= \sqrt{\dim \sigma} \int_G\widetilde f(h)
D^\sigma(kh^{-1} )\,dh = D^\sigma(k)\sqrt{\dim \sigma}
\int_G\widetilde f(h) D^\sigma(h^{-1} )\,dh=D^\sigma(k)\widehat
{\widetilde f}(\sigma)\ .
\end{array}
\end{equation}
Let us denote by $P_{\sigma,0}$ the projection of $H_\sigma$ onto
$H_{\sigma,0}$, so that $\widehat {\widetilde
f}(\sigma)=P_{\sigma,0}\widehat {\widetilde f}(\sigma)$,
and $\widehat G_0$ the set of irreducible unitary
representations of $G$ whose restriction to $K$ contains the trivial
representation.
If $\sigma\in \widehat G_0$ let us consider a basis of $H_\sigma$
such that the elements $\lbrace v_{p+1}, \dots , v_{\dim\sigma} \rbrace$,
for some  integer $p\ge 0$, span $H_{\sigma,0}$. Then the first $p$ rows
of the representative matrix of $\widehat {\widetilde f}(\sigma)$ in this
basis contain only zeros. Actually, by
(\ref{proiezione triviale}) and $P_{\sigma,0}$ being self-adjoint,
for $i\le p$
$$
\widehat {\widetilde f}_{i,j}(\sigma) = \langle \widehat {\widetilde
f}(\sigma)  v_j, v_i \rangle = \langle P_{\sigma,0} \widehat
{\widetilde f}(\sigma)   v_j, v_i \rangle=\langle \widehat
{\widetilde f}(\sigma)   v_j, P_{\sigma,0}v_i \rangle= 0\ .
$$
Recall that
a function $f\in L^2(G)$ is \emph{bi-$K$-invariant} if for every $g\in G, k_1, k_2\in K$
\begin{equation}\label{bicappa}
f(k_1gk_2)=f(g)\ .
\end{equation}
Condition \paref{bicappa} immediately entails that, for every
$k_1,k_2\in K$, $\sigma\in\widehat G$,
$$
\widehat f(\sigma)=D^\sigma(k_1)\widehat f(\sigma)D^\sigma(k_2)
$$
and it is immediate that a function $f\in L^2(G)$ is {bi-$K$-invariant} if and only if
for every $\sigma\in \hat G$
\begin{equation}\label{fourier-bicappa}
\widehat f(\sigma) = P_{\sigma,0} \widehat f(\sigma) P_{\sigma,0}\ .
\end{equation}
Now we briefly focus on the case $\X=\cS^2$ under the action of $G=SO(3)$ recalling basic facts we will need in the sequel (see \cite{faraut}, \cite{dogiocam} e.g. for further details). The isotropy
group $K\cong SO(2)$ of the north pole  is abelian, therefore its unitary irreducible representations are unitarily equivalent to its linear characters which we shall denote  $\chi_s, s\in \Z$, throughout the whole paper.

A complete set of unitary irreducible matrix representations of $SO(3)$  is given by  the so-called Wigner's $D$ matrices
$\lbrace D^\ell, \ell \ge 0 \rbrace$, where each $D^\ell(g)$ has dimension $(2\ell+1)\times (2\ell+1)$ and acts on a representative space that we shall denote $H_\ell$. The restriction to $K$ of each $D^\ell$ being unitarily equivalent to the direct sum of the representations $\chi_m$, $m=-\ell, \dots, \ell$, we can suppose
$v_{-\ell}, v_{-\ell +1}, \dots, v_\ell$ to be an orthonormal basis for $H_\ell$ such that
for every $m : |m| \le \ell$
\begin{equation}\label{restrizione}
D^\ell(k)v_m = \chi_m(k)v_m\ , \qquad k\in K\ .
\end{equation}
Let $D^\ell_{m,n}=\langle D^\ell v_n, v_m \rangle$ be the $(m,n)$-th entry of $D^\ell$ with respect to the basis fixed above.
It follows from (\ref{restrizione})  that for every $g\in SO(3), k_1, k_2\in K$,
\begin{equation}\label{prop fnz di Wigner}
D^\ell_{m,n}(k_1gk_2) = \chi_m(k_1)D^\ell_{m,n}(g)\chi_n(k_2)\ .
\end{equation}
The functions on $SO(3)$ $\lbrace g\to D^\ell_{m,n}(g), \, \ell\ge 0, m,n=-\ell,\dots,\ell \rbrace$
are usually called Wigner's $D$ functions.

Given $f\in L^2(SO(3))$,  its $\ell$-th Fourier coefficient defined in (\ref{Fourier coefficient}) is
\begin{equation}\label{coefficiente ellesimo}
\widehat f(\ell) := \sqrt{2\ell+1}\,\int_{SO(3)} f(g)
D^\ell(g^{-1})\,dg
\end{equation}
and
(\ref{PW1}) becomes
\begin{equation}\label{PW SO(3)}
f(g)=\sum_{\ell \ge 0} \sqrt{2\ell + 1} \sum_{m.n=-\ell}^{\ell
} \widehat f(\ell)_{n,m} D^\ell_{m,n}(g)\ .
\end{equation}
If $\widetilde f$ is the pullback  of  $f\in L^2(\cS^2)$, (\ref{restrizione}) entails that for every $\ell\ge 0$
$$
\widehat{\widetilde f}(\ell)_{n,m} \ne 0 \iff n= 0\ .
$$
Moreover if $f$ is left-$K$-invariant, then
$$
\widehat{\widetilde f}(\ell)_{n,m} \ne 0 \iff n,m = 0\ .
$$
In other words, an orthogonal basis for the space of the square integrable right-$K$-invariant functions on $SO(3)$ is given by the central columns  of the matrices $D^\ell$, $\ell \ge 0$. Furthermore
the subspace of the bi-$K$-invariant functions is spanned by the  central functions $D^\ell_{0,0}(\cdot),\, \ell \ge 0$, which are also \emph{real-valued}.

The important role of the other columns of Wigner's $D$ matrices will appear in \S 6.

For every $\ell \ge 0, m=-\ell \dots, \ell$,  let
\begin{equation}\label{armoniche sferiche1}
Y_{\ell,m}(x) := \sqrt{\frac{2\ell+1}{4\pi}}\, \overline{D^\ell_{m,0}(g_x)}\ , \qquad x\in \cS^2\ ,
\end{equation}
where $g_x$ is any rotation mapping the north pole of the sphere to $x$. This is a good definition thanks to the
invariance of each $D^\ell_{m,0}(\cdot)$ under the right action of $K$.
The functions in (\ref{armoniche sferiche1}) form an orthonormal basis of the space $L^2(\cS^2)$ considering the sphere with total mass equal to $4\pi$.  They
are usually known as \emph{spherical harmonics}.

By the previous discussion the Fourier expansion of a left-$K$-invariant function $f\in L^2(\cS^2)$ is
\begin{equation}\label{PW invariant sphere}
f=\sum_\ell \beta_\ell Y_{\ell,0}\ ,
\end{equation}
where $\beta_\ell := \int_{\cS^2} f(x) Y_{\ell,0}(x)\,dx$. The functions $Y_{\ell,0},\,\ell\ge 0$ are called  \emph{central spherical harmonics}.

\section{Isotropic random fields and positive definite functions}
Let $T=(T_x)_{x \in \X}$ be a  random field on the $G$-homogeneous space
$\X$, i.e. a collection of complex-valued r.v.'s defined
on some probability space $(\Omega, \F, \P)$, such that the map $
(\omega, x) \mapsto T_x(\omega) $ is $\F \otimes \B(\X)$-measurable.

$T$ is said to be \emph{a.s. continuous} if the functions $\X\ni
x\mapsto T_x$ are continuous a.s. $T$ is said to be \emph{second
order} if $T_x\in L^2(P)$ for every $x\in \X$. $T$ is  \emph{a.s.
square integrable} if
\begin{equation}\label{square integrable}
\int_\X |T_x|^2\,dx < +\infty,\qquad a.s.\ ,
\end{equation}
i.e. if the function $x\mapsto T_x$ belongs to $L^2(\X)$ a.s. In this
case, for every $f\in L^2(\X)$, we can consider the integral
$$
T(f):=\int_\X T_x \overline{f(x)}\,dx
$$
which defines a r.v. on $(\Omega, \F, \P)$.
For every $g\in G$ let  $T^g$ be the \emph{rotated field} defined as
$$
T^g_x:=T_{gx},\qquad x\in \X\ .
$$
\begin{definition}\label{invarian}
An a.s. square integrable random field $T$ on the homogeneous space
$\X$ is said to be (strict sense) \emph{$G$-invariant} or
\emph{isotropic} if and only if
the joint laws of
\begin{equation}\label{l2-invar}
(T(f_1),\dots,T(f_m))\qquad\mbox{and}\qquad (T(L_gf_1),\dots,T(L_gf_m))=
(T^g(f_1),\dots,T^g(f_m))
\end{equation}
coincide for every $g\in G$ and $f_1,f_2,\dots,f_m\in L^2(\X)$.
\end{definition}
This definition is somehow different from the one usually considered in the literature, where
the requirement is the equality of the finite dimensional distributions, i.e. that the random vectors
\begin{equation}\label{invar-continui1}
(T_{x_1},\dots,T_{x_m})\qquad\mbox{and}\qquad(T_{gx_1},\dots,T_{gx_m})
\end{equation}
have the same law for every choice of $g\in G$ and
$x_1,\dots,x_m\in \X$. Remark that (\ref{l2-invar})
implies (\ref{invar-continui1}) (see \cite{mp-continuity}) and
that, conversely, by standard approximation arguments
(\ref{invar-continui1}) implies (\ref{l2-invar}) if $T$ is
continuous.

To every a.s. square integrable random field $T$ on the group $G$ we
can associate the set of operator-valued r.v.'s $(\widehat
{T}(\sigma))_{\sigma\in \widehat G}$ defined ``pathwise'' as
\begin{equation}
\widehat { T}(\sigma) = \sqrt{\dim\sigma}\int_G T_g D^\sigma(g^{-1})\,dg\ .
\end{equation}
From  (\ref{PW for compact groups}) therefore
\begin{equation}
T_g = \sum_{\sigma\in \widehat G} \sqrt{\dim\sigma} \tr(\widehat {
T}(\sigma) D^\sigma(g))\ ,
\end{equation}
where the convergence takes place in $L^2(G)$ a.s.

The following statement points out a remarkable property that is
enjoyed by every isotropic and second order random field (see
\cite{mp-continuity}).
\begin{prop}\label{mean square continuous}
Every a.s. square integrable, isotropic and second order random field $T$
on the homogeneous space $\X$ of the compact group $G$
is \emph{mean square continuous}, i.e.
\begin{equation}
\lim_{y\to x} \E[ |T_y - T_x|^2] = 0\ .
\end{equation}
\end{prop}
Assume that $T$ is (i) a.s. square integrable, (ii) isotropic and (iii) second order.
Assume also that
$T$ is centered, which can be done without loss of generality. The
\emph{covariance kernel} $R:\X\times \X \to \C$ of $T$ is defined as
$$
R(x,y)=\Cov(T_x,T_y)=\E[T_x\overline{T_y}]\ .
$$
As $T$ is isotropic, $R$ is $G$-invariant, i.e. $R(gx, gy)=R(x,y)$ for
every $g\in G$, and also continuous thanks to Proposition \ref{mean
square continuous}. To $R$ we can associate the function
$\phi:G\to\C$
\begin{equation}\label{funzione fi}
\phi(g):=R(gx_0,x_0)\ .
\end{equation}
Such a function $\phi$ is

$\bullet$ continuous, as a consequence of the continuity of $R$.

$\bullet$ Bi-$K $-invariant: for every $k_1,k_2\in K $ and $g\in G$ we have,
using the invariance of $R$ and the fact that $k_ix_0=x_0, i=1,2$,
$$
\phi(k_1gk_2)=R(k_1gk_2x_0,x_0)=R(k_1gx_0,x_0) =R(gx_0,k_1^{-1}x_0)=R(gx_0,x_0)=\phi(g)\ .
$$

$\bullet$ Positive definite: as $R$ is a positive definite kernel, for every $g_1,\dots,g_m\in G$, $\xi_1,\dots,\xi_m\in \C$,
\begin{equation}\label{pd-for-phi}
\sum_{i,j} \phi(g_i^{-1}g_j)\overline{\xi_i} \xi_j =
\sum_{i,j} R(g_i^{-1}g_jx_0,x_0)\overline{\xi_i} \xi_j=
\sum_{i,j} R(g_jx_0,g_ix_0)\overline{\xi_i} \xi_j\ge 0\ .
\end{equation}
By standard approximation arguments \paref{pd-for-phi} is equivalent to

\begin{equation}\label{dis1}
\int_G \int_G \phi(h^{-1}g) f(h) \overline{f(g)}\,dg\,dh \ge 0
\end{equation}
for every continuous function $f$.
Finally $\phi$ determines the covariance kernel $R$: if $g_x
x_0=x, g_yx_0=y$, then
$$
R(x,y)=R(g_xx_0,g_yx_0)=R(g_y^{-1}g_xx_0,x_0)=\phi(g_y^{-1}g_x)\ .
$$
For a function $\zeta$ on $G$ let
$$
\breve \zeta(g):=\overline{\zeta(g^{-1})}\ .
$$
Recall that every positive definite function $\phi$ on $G$ satisfies (see \cite{sugiura} p.123 e.g.)
\begin{equation}\label{phi-phibreve}
\breve\phi(g)=\phi(g)\ .
\end{equation}
\begin{remark}\rm\label{aggiunto}
If $\zeta\in L^2(G)$, then for every $\sigma \in \widehat G$ we have
$\widehat {\breve \zeta} (\sigma)= \widehat \zeta(\sigma)^*$.
Actually,
$$
\dlines{
\widehat {\breve \zeta} (\sigma)
=\sqrt{\dim\sigma} \int_G \overline{\zeta(g^{-1})} D^\sigma(g^{-1})\,dg=\sqrt{\dim\sigma}\int_G
\overline{\zeta(g)} D^\sigma(g)\,dg=\cr
=
\sqrt{\dim\sigma}\int_G \bigl(\zeta(g)
D^\sigma(g^{-1})\bigr)^*\,dg=\widehat \zeta(\sigma)^*\ .\cr }
$$
\end{remark}
\qed

The following proposition states some (not really unexpected)
properties of positive definite functions that we
shall need later.
\begin{prop}\label{structure of positive definite}
Let $\phi$ be a continuous positive definite function and $\sigma\in \widehat G$.

a) The operator coefficient $\widehat\phi(\sigma):H_\sigma\to H_\sigma$  as defined in (\ref{Fourier coefficient}) is Hermitian positive definite.

b) Let $ \phi^\sigma:G\to \C$ be the $\sigma$ component of $\phi$ defined in (\ref{component}). Then $\phi^\sigma$ is also positive definite.
\end{prop}
\begin{proof}
a) For a fixed a basis $v_1,\dots,v_{\dim\sigma}$ of $H_\sigma$, we have by the invariance of the Haar measure
$$
\dlines{
\langle \widehat \phi(\sigma) v,v \rangle =\sqrt{\dim\sigma}\int_G \phi(g)
\langle D^\sigma(g^{-1})v,v\rangle\,dg = \sqrt{\dim\sigma}\int_G \int_G \phi(h^{-1}g)
\langle D^\sigma(g^{-1}h)v,v\rangle\,dg\,dh=\cr =
\sqrt{\dim\sigma}\int_G \int_G \phi(h^{-1}g) \langle D^\sigma(h)v,D^\sigma(g)v\rangle\,dg\,dh=\cr
=\sqrt{\dim\sigma}\int_G \int_G \phi(h^{-1}g) \sum_k ({D^\sigma(h)}{v})_k\overline{({D^\sigma(g)} {v})}_k\,dg\,dh =\cr
=\sqrt{\dim\sigma}\sum_k \int_G \int_G \phi(h^{-1}g) f_k(h)\overline{f_k(g)}\,dg\,dh \ge 0\ ,\cr
}
$$
where, for every $k$, $f_k(g):=({D^\sigma(g)}{v})_k$ and we can conclude thanks to (\ref{dis1}).

b) By the
Peter-Weyl theorem \paref{PWdiverso}
\begin{equation}
\phi= \sum_{\sigma\in \widehat G} \phi^\sigma
\end{equation}
in $L^2(G)$. Let $f\in L^2_\sigma(G)$ as in (\ref{dis1}) and recall
that $f$ is a continuous function. We have
$$
\dlines{ \int_G \int_G \phi(h^{-1}g) f(h)
\overline{f(g)}\,dg\,dh=\sum_{\sigma'\in\widehat G}\int_G \int_G
\phi^{\sigma'}(h^{-1}g) f(h) \overline{f(g)}\,dg\,dh=\cr =\int_G
\sum_{\sigma'\in\widehat G} \underbrace{\int_G
\phi^{\sigma'}(h^{-1}g) f(h)\,dh}_{= f\ast \phi^{\sigma'}(g)}
\overline{f(g)}\,dg= \int_G \int_G \phi^{\sigma}(h^{-1}g)
f(h)\overline{f(g)}\,dg\,dh\cr }
$$
as $f\ast \phi^{\sigma'} \ne 0$ if and
only if $\sigma'=\sigma$. Therefore for every $\sigma\in \widehat G$ and $f\in L^2_\sigma(G)$
\begin{equation}
\int_G \int_G \phi^{\sigma}(h^{-1}g) f(h) \overline{f(g)}\,dg\,dh=\int_G \int_G \phi(h^{-1}g) f(h) \overline{f(g)}\,dg\,dh \ge 0\ .
\end{equation}
Let now $f\in L^2(G)$ and $f=\sum_{\sigma'} f^{\sigma'}$ be its Fourier series. The same argument as above gives
$$
\int_G \int_G \phi^{\sigma}(h^{-1}g) f(h) \overline{f(g)}\,dg\,dh=\int_G \int_G \phi^{\sigma}(h^{-1}g) f^\sigma(h) \overline{f^\sigma(g)}\,dg\,dh\ge 0
$$
so that $\phi^\sigma$ is a positive definite function.
\end{proof}
Another important property enjoyed by positive definite functions on $G$ is shown
in the following classical theorem (see \cite{bib:G}, Theorem 3.20 p.151).
\begin{theorem}\label{gangolli}
Let $\zeta$ be a continuous positive definite function on $G$ and denote $\zeta^\sigma$
the component of $\zeta$ on the $\sigma$-isotypical subspace $L^2_\sigma(G)$. Then
\begin{equation}\label{gangolli-true}
\sum_{\sigma\in\widehat G}\sqrt{\dim \sigma}\,\tr \widehat\zeta(\sigma)<+\infty
\end{equation}
and the Fourier series
$$
\zeta=\sum_{\sigma\in \widehat G} \zeta^\sigma
$$
converges uniformly on $G$.
\end{theorem}
We shall need the following ``square root'' theorem in the proof of the representation formula of Gaussian isotropic random fields on $\X$.
\begin{theorem}\label{square-root} Let $\phi$ be a bi-$K$-invariant positive
definite continuous function on $G$. Then there exists a  bi-$K$-invariant function $f\in L^2(G)$ such that $\phi=f*\breve
f$. Moreover, if $\phi$ is real valued then $f$ also can be chosen
to be real valued.
\end{theorem}
\begin{proof}
For every $\sigma\in\widehat G$, $\widehat \phi(\sigma)$ is
Hermitian positive definite. Therefore there exist matrices
$\Lambda(\sigma)$ such that
$\Lambda(\sigma)\Lambda(\sigma)^*=\widehat\phi(\sigma)$.
Let
$$
f= \sum_{\sigma\in\widehat G} \underbrace{\sqrt{\dim\sigma}\,
\tr\bigl(\Lambda(\sigma) D^\sigma\bigr)}_{=f^\sigma}\ .
$$
This actually defines a function $f\in L^2(G)$ as it is easy to see that
$$
\Vert f^\sigma\Vert^2_2=\sum_{i,j=1}^{\dim\sigma}(\Lambda(\sigma)_{ij})^2=
\tr(\Lambda(\sigma)\Lambda(\sigma)^*)=\tr(\widehat\phi(\sigma))
$$
so that
$$
\Vert f\Vert^2_2=\sum_{\sigma\in\widehat G}\Vert f^\sigma\Vert^2_2=
\sum_{\sigma\in\widehat G}\tr(\widehat\phi(\sigma))<+\infty
$$
thanks to \paref{gangolli-true}.
By Remark \ref{aggiunto} and (\ref{conv1}), we have
$$
\phi= f\ast \breve f\ .
$$
Finally the matrix $\Lambda(\sigma)$ can be chosen to be Hermitian
and with this choice $f$ is bi-$K $-invariant
 as the relation \paref{fourier-bicappa}
$\widehat f (\sigma)=P_{\sigma,0}\widehat f(\sigma)P_{\sigma,0}$
still holds.

The last statement follows from Proposition \ref{real-sq} in the Appendix.

\end{proof}
Note that the decomposition of Theorem \ref{square-root} is not
unique, as the Hermitian square root of the positive definite
operator $\widehat\phi(\sigma)$ is not unique itself (see Remark \ref{rem-sfera}
below for explicit examples).
\section{Construction of Gaussian isotropic random fields}\label{sec4}

In this section we develop a method of constructing isotropic
Gaussian random fields on $\X$. It is pretty much inspired by P.
L\'evy's white noise construction of his spherical Brownian motion
(\cite{bib:L}).
Let $(X_n)_n$ be a sequence of i.i.d. $N(0,1)$-distributed r.v.'s on some probability space $(\Omega, \F, \P)$ and  denote by $\mathscr{H}\subset L^2(\P)$ the real Hilbert space generated by $(X_n)_n$.
Let $(e_n)_n$ be an orthonormal basis of $L^2_{\mathbb{R}}(\mathscr{X})$.
We define an isometry $S:L^2_{\mathbb{R}}(\mathscr{X})\to\mathscr{H}$ by
$$
L^2_{\mathbb{R}} (\mathscr{X})\ni \sum_k \alpha_k
e_k\enspace\leftrightarrow\enspace \sum_k \alpha_k X_k \in
\mathscr{H}\ .
$$
It is easy to extend $S$ to an isometry on  $L^2(\mathscr{X})$,
indeed if $f\in L^2(\mathscr{X})$, then $f=f_1+if_2$, with $f_1, f_2
\in L^2_{\mathbb{R}}(\mathscr{X})$, hence just set
$S(f)=S(f_1)+iS(f_2)$. Such an isometry respects the real character
of the function $f\in L^2(\X)$ (i.e. if $f$ is real then $S(f)$ is a
real r.v.).

Let $f$ be a left $K $-invariant function in $L^2(\mathscr{X})$.
We then define a random field $(T^f_x)_{x\in \mathscr{X}}$
associated to $f$ as follows: set $T^f_{x_0}=S(f)$ and, for every $x\in \X$,
\begin{equation}\label{def campo}
T^f_x =S(L_gf)\ ,
\end{equation}
where $g\in G$ is such that $gx_0=x$ ($L$ still denotes the left
regular action of $G$).
This is a good definition: in fact if also $\widetilde{g}\in G$ is such that
$\widetilde{g}x_0=x$, then $\widetilde{g}=gk$ for some $k\in K$ and therefore $L_{\widetilde g}f(x)=f(k^{-1}g^{-1}x)=
f(g^{-1}x)=L_gf(x)$ so that
$$
S(L_{\widetilde{g}}f)=S(L_gf)\ .
$$
The random field $T^f$ is mean square integrable, i.e.
$$
\dlines{ \E \Bigl[\int_\X |T^f_x|^2\,dx \Bigr]
< +\infty\ .}
$$
Actually,
if $g_x$ is any element of $G$ such that $g_xx_0=x$ (chosen in some
measurable way), then, as $\E[|T^f_x|^2]=\E [|S(L_{g_x} f)|]^2=\|
L_{g_x}f \|^2_{L^2(\X)}=\| f \|^2_{L^2(\X)}$, we have $\E \int_\X |T^f_x|^2\,dx= \| f \|^2_{L^2(\X)}$.
$T^f$ is a centered and \emph{complex-valued Gaussian} random field.
Let us now check that $T^f$ is isotropic. Recall that
the law of a complex-valued Gaussian random vector $Z=(Z_1, Z_2)$ is
completely characterized by its mean value $\E[Z]$, its covariance
matrix $\E[ \left(Z - \E[Z]\right) \left( Z- \E[Z] \right)^*]$ and
the \emph{pseudocovariance} or \emph{relation matrix} $\E[ \left(Z -
\E[Z]\right) \left( Z- \E[Z] \right)']$. We have

(i) as $S$ is an isometry
$$
\displaylines{
\mathbb{E}[T^f_{gx}\overline{T^f_{gy}}]=\mathbb{E}[S(L_{gg_x}f)\overline{S(L_{gg_y}f)}]=
\langle L_{gg_x }f, L_{gg_y}f\rangle_{L^2(\X)}
=\langle L_{g_x }f, L_{g_y}f \rangle_{L^2(\X)}=
\mathbb{E}[T^f_x\overline{T^f_y}]\ .}
$$

(ii) Moreover, as complex conjugation commutes both with $S$ and the
left regular representation of $G$,
$$
\displaylines{
\mathbb{E}[T^f_{gx}T^f_{gy}]=\mathbb{E}[S(L_{gg_x}f)\overline{S(L_{gg_y}\overline{f})}]=
\langle L_{gg_x }f, L_{gg_y}\overline{f}\rangle_{L^2(\X)}=
\langle L_{g_x }f, L_{g_y}\overline{f}\rangle_{L^2(\X)}=\mathbb{E}[T^f_xT^f_y]\ .}
$$
Therefore $T^f$ is isotropic because it has the same covariance and
relation kernels as the rotated field $(T^f)^g$ for every $g\in G$.

If $R^f(x,y)=\E[T^f_x \overline{T^f_y}]$ denotes its covariance
kernel, then the associated positive definite function
$\phi^f(g):=R(gx_0,x_0)$ satisfies
\begin{equation}\label{convolution for phi}
\begin{array}{c}
\displaystyle\phi^f(g)=\E[S(L_g f)\overline{S(f)}]=
\langle L_gf, f\rangle
=\\
\noalign{\vskip3pt}
\displaystyle= \int_G \widetilde f(g^{-1}h) \overline{\widetilde f(h)}\,dh= \int_G \widetilde f(g^{-1}h) \breve {\widetilde f} (h^{-1})\,dh=\widetilde f \ast \breve {\widetilde f} (g^{-1})\ , \\
\end{array}
\end{equation}
where $\widetilde f$ is the pullback on $G$ of $f$ and the convolution $\ast$ is in $G$. Moreover the relation function of $T^f$
$
\zeta^f(g) := \E[T^f_{gx_0} T^f_{x_0}]
$
satisfies
\begin{equation}\label{convolution for zeta}
\zeta^f(g)=\E[S(L_gf)S(f)]=\langle L_gf, \overline{f}\rangle\ .
\end{equation}
One may ask whether every a.s. square integrable, isotropic, complex-valued Gaussian centered random field on $\X$ can be obtained with this construction: the answer
is \emph{no} in general. It is however positive if we consider
{\it real} isotropic Gaussian random fields (see Theorem \ref{real-general} below). Before considering the case of a general homogeneous space $\X$, let us look first at the case of the sphere, where things are particularly simple.

\begin{remark}\label{rem-sfera} \rm (Representation of
real Gaussian isotropic random fields on $\cS^2$) If $\X=\cS^2$
under the action of $SO(3)$, every isotropic, \emph{real} Gaussian
and centered random field is of the form \paref{def campo} for some
left-$K$-invariant function $f:\cS^2\to \R$. Indeed let us consider on $L^2(\cS^2)$ the Fourier
basis $Y_{\ell,m}$, $\ell=1,2,\dots$, $m=-\ell,\dots,\ell$,
given by the spherical harmonics (\ref{armoniche sferiche1}).
Every continuous positive definite left-$K$-invariant function $\phi$ on $\cS^2$ has a Fourier expansion of the form (\ref{PW invariant sphere})
\begin{equation}\label{fi per sfera}
\phi = \sum_{\ell \ge 0} \alpha_\ell Y_{\ell,0}\ ,
\end{equation}
where (Proposition
\ref{structure of positive definite}) $\alpha_\ell \ge 0$ and
$$
\sum_{\ell \ge 0} \sqrt{2\ell+1}\,\alpha_\ell<+\infty
$$
(Theorem \ref{gangolli}).
The $Y_{\ell,0}$'s being real, the function $\phi$ in (\ref{fi per sfera})
is\emph{ real}, so that, by \paref{phi-phibreve}, $\phi(g)=\phi(g^{-1})$
(in this remark and in the next example we identify functions on $\cS^2$ with
their pullbacks on $SO(3)$ for simplicity of notations).

If $\phi$ is the positive definite left-$K$-invariant function associated to
$T$, then, keeping in mind that $Y_{\ell,0}*Y_{\ell',0}=(2\ell+1)^{-1/2}
Y_{\ell,0}\delta_{\ell,\ell'}$, a  square root $f$ of $\phi$ is given by
\begin{equation}
f = \sum_{\ell \ge 0} \beta_\ell \, Y_{\ell,0}\ ,
\end{equation}
where $\beta_\ell$ is a complex number such that
$$
\frac{|\beta_\ell |^2}{\sqrt{2\ell+1}}= \sqrt{\alpha_\ell}\ .
$$
Therefore there exist infinitely many real
functions $f$ such that $\phi(g)=\phi(g^{-1})=
f \ast \breve{f}(g)$, corresponding to the choices $\beta_\ell=\pm
( (2\ell+1)\alpha_\ell )^{1/4}$. For each of these, the random field $T^f$ has
the same distribution as $T$, being real and having the same associated positive
definite function.
\end{remark}
\qed

\begin{example}\rm{(P.L\'evy's spherical Brownian field)}\label{MB}.
Let us choose as a particular instance of the previous construction
$f=c1_{H}$, where $H$ is the half-sphere centered at the north pole
$x_0$ of $\cS^2$ and $c$ is some constant to be chosen later.

Still denoting by $S$ a white noise on $\cS^2$, from (\ref{def campo})
we have
\begin{equation}
T^f_x = c S(1_{H_x})\ ,
\end{equation}
where $1_{H_x}$ is the half-sphere centered at $x\in \cS^2$. Now,
let $x, y\in \bS^2$ and denote by $d(x,y) = \theta$ their distance, then,
$S$ being an isometry,
\begin{equation}
\Var(T^f_x - T^f_y) = c^2 \| 1_{H_x\vartriangle H_y}\|^2\ .
\end{equation}
The symmetric difference $H_x\vartriangle H_y$ is formed by the union of
two wedges whose total surface is equal to $\frac{\theta}{\pi}$
(recall that we
consider the surface of $\cS^2$ normalized with total mass $=1$).
Therefore, choosing $c= \sqrt{\pi}$, we have
\begin{equation}
\Var(T^f_x - T^f_y) = d(x,y)
\end{equation}
and furthermore $\Var(T^f_x) = \tfrac\pi2$.
Thus
\begin{equation}\label{aa}
\Cov(T^f_x, T^f_y) = \tfrac12 \bigl( \Var(T^f_x) + \Var(T^f_y) -
\Var(T^f_x - T^f_y)\bigr) = \tfrac\pi2 - \tfrac12 d(x,y)\ .
\end{equation}
Note that the positive definiteness of (\ref{aa}) implies that the distance $d$
is a Schoenberg restricted negative definite kernel on $\cS^2$. The random
field $W$
\begin{equation}
W_x := T^f_x - T^f_{x_0}
\end{equation}
is \emph{P.L\'evy's spherical Brownian field}, as its
covariance kernel is
\begin{equation}\label{kernel del mb}
\Cov(W_x, W_y) = \tfrac12 \left( d(x,x_0) + d(y,x_0) - d(x,y) \right)\ .
\end{equation}
In particular the kernel at the r.h.s. of (\ref{kernel del mb})
is positive definite (see also \cite{bib:G}).
Let us compute the expansion into spherical harmonics of the
positive definite function $\phi$ associated to the random field
$T^f$ and to $f$.  We have $\phi(x)=\frac\pi2-\frac
12\,d(x,x_0)$, i.e.  $\phi(x)=\frac\pi2-\frac12\, \vt$ in spherical coordinates, $\vt$ being the colatitude of $x$, whereas
$Y_{\ell,0}(x)=\sqrt{2\ell+1}\,P_\ell(\cos\vt)$ where $P_\ell$ is the $\ell$-th Legendre polynomial. This formula for the central spherical harmonics
differs slightly from the usual one, as we consider the total
measure of $\cS^2$ to be $=1$. Then, recalling the normalized
measure of the sphere is $\frac 1{4\pi}\,\sin\vt\, d\vt\, d\phi$ and
that $Y_{\ell,0}$ is orthogonal to the constants
$$
\dlines{ \int_{\cS^2}\phi(x)Y_{\ell,0}(x)\, dx=-\frac
14\,\sqrt{2\ell+1}\int_0^\pi\vt P_\ell(\cos\th)\sin \vt\, d\vt=
-\frac 14\,\sqrt{2\ell+1}\int_{-1}^{1}\arccos t\, P_\ell(t)\, dt=\cr
=\frac 14\,\sqrt{2\ell+1}\int_{-1}^{1}\arcsin t\, P_\ell(t)\,
dt=\frac14\,\sqrt{2\ell+1}\,c_\ell\ , \cr }
$$
where
$$
c_\ell=\pi\Bigl\{\frac{3\cdot 5\cdots(\ell-2)}{2\cdot 4\cdots
(\ell+1))}\Bigr\}^2\,\quad \ell=1,3,\dots
$$
and $c_\ell=0$ for $\ell$ even (see \cite{MR1424469}, p.~325). As for
the function $f=\sqrt{\pi}\,1_{H}$, we have
$$
\int_{\cS^2}f(x)Y_{\ell,0}(x)\, dx=\frac {\sqrt{\pi}}2
\,\sqrt{2\ell+1}\int_0^{\pi/2}P_\ell(\cos\vt)\sin\vt\, d\vt= \frac
{\sqrt{\pi}}2 \,\sqrt{2\ell+1}\int_0^1P_\ell(t)\, dt\ .
$$
The r.h.s. can be computed using Rodrigues formula for the
Legendre polynomials (see again \cite{MR1424469}, p.~297) giving
that it vanishes for $\ell$ even and equal to
\begin{equation}\label{2m}
(-1)^{m+1}\,\frac {\sqrt{\pi}}2 \,\sqrt{2\ell+1}\,\frac
{(2m)!{2m+1\choose m}}{2^{2m+1}(2m+1)!}
\end{equation}
for $\ell=2m+1$. Details of this computation are given in Remark
\ref{rod} in the Appendix. Simplifying the factorials the previous
expression becomes
$$
\dlines{ (-1)^m\,\frac {\sqrt{\pi}}2 \,\sqrt{2\ell+1}\,\frac
{(2m)!}{2^{2m+1}m!(m+1)!}=(-1)^m\,\frac {\sqrt{\pi}}2
\,\sqrt{2\ell+1}\,\frac{3\cdots (2m-1)}{2\cdots (2m+2)}
=(-1)^m\,\frac 12 \,\sqrt{2\ell+1}\,\sqrt{c_\ell}\ .\cr }
$$
Therefore the choice $f=\sqrt{\pi}\, 1_H$ corresponds to taking
alternating signs when taking the square roots. Note that the
choice $f'=\sum_\ell \beta_\ell Y_{\ell,0}$ with $\beta_\ell=\frac 12
\,\sqrt{2\ell+1}\,\sqrt{c_\ell}$ would have given a function
diverging at the north pole $x_0$. Actually it is elementary to
check that the series $\sum_\ell ( {2\ell+1})\,\sqrt{c_\ell}$
diverges so that $f'$ cannot be continuous by Theorem \ref{gangolli}.
\end{example}
\qed

The result of Remark \ref{rem-sfera} concerning $\cS^2$ can be extended to the case of a general homogeneous space.
\begin{theorem}\label{real-general} Let $\X$ be the homogeneous space of a compact group $G$ and let $T$ be an a.s. square
 integrable isotropic Gaussian \emph{real} random field on $\X$.
 Then there exists a left-$K$-invariant function $f\in L^2(\X)$ such that $T^f$ has the same distribution as $T$.
\end{theorem}
\begin{proof} Let $\phi$ be the invariant positive definite function associated to $T$.
Thanks to \paref{convolution for phi} it is sufficient to prove that
there exists a \emph{real} $K$-invariant
function $f\in L^2(\X)$ such that $\phi(g)=\widetilde f \ast \breve{\widetilde
f}(g^{-1})$. Keeping in mind that $\phi(g)=\phi(g^{-1})$, as $\phi$
is real and thanks to \paref{phi-phibreve}, this follows from
Theorem \ref{square-root}.

\end{proof}

As remarked above $f$ is not unique.

Recall that a complex valued Gaussian r.v. $Z=X+iY$ is said to be
\emph{complex Gaussian} if the
r.v.'s $X,Y$ are jointly Gaussian, are independent and have the same variance. A $\C^m$-valued r.v.
$Z=(Z_1,\dots, Z_m)$ is said to be complex Gaussian if the r.v.
$\alpha_1Z_1+\dots+\alpha_mZ_m$ is complex Gaussian for every choice of $\alpha_1,\dots,\alpha_m\in \C$.
\begin{definition}An a.s. square integrable random field $T$ on $\X$ is said to be \emph{complex Gaussian}
if and only if the complex valued r.v.'s
$$
\int_\X T_xf(x)\, dx
$$
are complex Gaussian for every choice of $f\in L^2(\X)$.
\end{definition}
Complex Gaussian random fields will play an important role in the next sections. By
now let us remark
that, in general, it is not possible to obtain a complex Gaussian random field by the procedure
\paref{def campo}.

\begin{prop}\label{zeta-prop} Let $\zeta(x,y)=\E[T_xT_y]$ be the relation kernel of a centered complex
Gaussian random field $T$. Then $\zeta\equiv 0$.
\end{prop}
\begin{proof} It easy to check that a centered complex valued r.v. $Z$ is complex Gaussian if and only
if $\E[Z^2]=0$. As for every $f\in L^2(\X)$
$$
\int_\X\int_\X \zeta(x,y)f(x)f(y)\, dx dy=\E\Bigl[\Bigl(\int_\X T_xf(x)\, dx\Bigr)^2\Bigr]=0\ ,
$$
it is easy to derive
that $\zeta\equiv0$.
\end{proof}
Going back to the situation of Remark \ref{rem-sfera}, the relation function $\zeta$ of the
random field $T^f$ is easily found to be
\begin{equation}
\zeta^f=\sum_{\ell \ge 0} \beta_\ell^2\, Y_{\ell,0}\ .
\end{equation}
and cannot vanish unless $f\equiv 0$ and $T^f$ vanishes itself.
Therefore no isotropic complex Gaussian random field on the sphere can be obtained by
the construction \paref{def campo}.
\section{Random sections of vector bundles}
We now investigate the case of Gaussian isotropic spin random fields
on $\cS^2$, with the aim of extending the representation result
of Theorem \ref{real-general}. These models have
received recently much attention (see \cite{bib:LS}, \cite{bib:M} or \cite{dogiocam}), being motivated by the modeling of CMB data. Actually our point of view begins from \cite{bib:M}.

We consider first the case of a general vector bundle. Let $\xi= (E,
p, B)$ be a finite-dimensional \emph{complex vector bundle} on the
topological space $B$, which is called the \emph{base space}. The
surjective map
\begin{equation}
p: E\goto B
\end{equation}
is the
\emph{bundle projection}, $p^{-1}(x)$, $x\in B$ is the {\it fiber}
above $x$.
Let us denote $\B(B)$ the Borel $\sigma$-field of $B$.
A section of $\xi$ is a map $u: B \to E$ such that $p \circ u=id_B$. As $E$ is itself a topological space, we can speak of continuous sections.

We suppose from now on that every fibre $p^{-1}(x)$ carries an inner
product and a measure $\mu$ is given on the base space. Hence we can
consider square integrable sections, as those such that
$$
\int_B\langle u(x),u(x)\rangle_{p^{-1}(x)}\,d\mu(x)<+\infty
$$
and define the corresponding $L^2$ space accordingly.

Let $(\Omega, \F, \P)$ be a probability space.
\begin{definition}\label{definizione di sezione aleatoria}
A \emph{random section $T$
of the vector bundle $\xi$} is a collection of $E$-valued random variables
$(T_x)_{x\in B}$ indexed by the elements of the base space $B$ such that
the map $\Omega \times B \ni(\omega, x)  \mapsto T_x(\omega)$
is $\F \otimes \B(B)$-measurable and, for every $\omega$, the path
$$
B\ni x\to T_x(\omega)\in E
$$
is a section of $\xi$, i.e. $p\circ T_\cdot(\omega) =id_B$.
\end{definition}
Continuity of a random section $T$ is easily defined by requiring that
for every $\omega\in \Omega$ the map $x \mapsto T_x$
is a continuous section of $\xi$. Similarly a.s. continuity is defined.
A random section $T$ of $\xi$ is a.s. square integrable if  the map
$x \mapsto \| T_x (\omega)\|^2 _{p^{-1}(x)}$ is a.s. integrable, it is second order if $\E[ \| T_x \|^2_{p^{-1}(x)}] < +\infty$ for every $x\in B$ and
mean square integrable
if
$$
\E\Bigl[\int_B\| T_x \|^2 _{p^{-1}(x)} \, d\mu(x)\Bigr]< +\infty\ .
$$
As already remarked in $\cite{bib:M}$, in defining the notion of mean square continuity for
a random section, the naive approach
$$
\lim_{y\to x} \E [\| T_x - T_y \|^2] =0
$$
is not immediately meaningful as $T_x$ and $T_y$ belong to different
(even if possibly isomorphic) spaces (i.e. the fibers).
A similar difficulty arises for the definition of strict sense invariance w.r.t. the action of a topological group on the bundle.
We shall investigate these points below.

A case of particular interest to us are the homogeneous (or twisted) vector bundles. Let $G$ be a compact group, $K$ a closed subgroup and $\X=G/K$.
Given an irreducible unitary representation $\tau$ of $K$ on the complex (finite-dimensional) Hilbert space $H$,
$K$ acts on the Cartesian product $G\times H$ by the action
$$
k(g,z):= (gk, \tau(k^{-1})z)\ .
$$
Let   $G\times_\tau H=\lbrace \theta(g,z) : (g,z) \in G\times H
\rbrace$ denote the quotient space of the orbits $\theta(g,z) =
\lbrace (gk, \tau(k^{-1})z) : k\in K \rbrace$ under the above
action. $G$ acts on $G\times_\tau H$ by
\begin{equation}\label{action}
h \theta(g,z) := \theta(hg,z)\ .
\end{equation}
The map $G\times H \to \X: (g,z)\to gK$ is constant on the
orbits $\theta(g,z)$ and   induces the projection
$$
G\times_\tau H\ni\theta(g,z)\enspace\mathop{\to}^{\pi_\tau\,}\enspace gK\in \X
$$
which is a continuous $G$-equivariant map. $\xi_\tau=
(G\times_\tau H, \pi_\tau, \X)$ is a $G$-vector bundle: it is the \emph{homogeneous vector bundle associated to
the representation $\tau$}. The  fiber
$\pi^{-1}_\tau(x)$ is isomorphic to $H$ for every $x\in \X$ (see
\cite{B-D}). More precisely, for $x\in\X$ the fiber $\pi_\tau^{-1}(x)$
is the set of
elements of the form $\th(g,z)$ such that $gK=x$. We define the scalar
product of two such elements as
\begin{equation}\label{prod scalare}
\langle \th(g,z),\th(g,w)\rangle_{\pi_\tau^{-1}(x)}=\langle z,w\rangle_{H}
\end{equation}
for some fixed $g\in G$ such that $gK=x$, as it is immediate that this
definition does not depend on the choice of such a $g$.
Given a function $f:G \to H$ satisfying
\begin{equation}\label{funzioni di tipo W}
f(gk)=\tau(k^{-1})f(g)\ ,
\end{equation}
then to it we can associate the section of $\xi_\tau$
\begin{equation}\label{proiezione}
u(x)=u(gK)=\theta(g,f(g))
\end{equation}
as again this is a good definition, not depending
of the choice of $g$ in the coset $gK$. The functions $f$ satisfying to
(\ref{funzioni di tipo W}) are called right $K$-covariant functions
of type $\tau$ (\emph{functions of type $\tau$} from now on).

More interestingly, also the converse is true.
\begin{prop}\label{pullback-s-deterministic}
Given a section $u$ of $\xi_\tau$, there exists a unique function
$f$ of type $\tau$ on $G$ such that $u(x)=\theta(g,f(g))$ where
$gK=x$. Moreover $u$ is continuous if and only if
$f:G\to H$ is continuous.
\end{prop}
\begin{proof} Let $(g_x)_{x\in\X}$ be
a measurable selection such that
$g_xK=x$ for every $x\in\X$. If $u(x)=\theta(g_x, z)$, then define
$f(g_x):=z$ and extend the definition to the elements of the coset
$g_xK$ by $f(g_xk):= \tau(k^{-1})z$; it is easy to check that such a
$f$ is of type $\tau$, satisfies (\ref{proiezione}) and is the unique
function of type $\tau$ with this property.

Note that the continuity of $f$ is equivalent to the continuity of
the map
\begin{equation}\label{mappa1}
F: g\in G \to (g,f(g))\in G\times H\ .
\end{equation}
Denote $pr_1: G\to \X$ the canonical projection onto the quotient space $\X$
and $pr_2: G\times H \to G\times_\tau H$ the canonical projection
onto the quotient space $G\times_\tau H$. It is immediate that
\begin{equation*}\label{mappa2}
pr_2 \circ F = u \circ pr_1\ .
\end{equation*}
Therefore $F$ is continuous if and only if $u$ is continuous,
the projections $pr_1$ and $pr_2$ being continuous open mappings.

\end{proof}
We shall again call $f$ the \emph {pullback} of $u$.
Remark that, given two sections $u_1, u_2$ of $\xi_\tau$ and their respective pullbacks $f_1,f_2$, we have
\begin{equation}\label{prod_scalare}
\langle u_1, u_2 \rangle := \int_\X \langle u_1(x),
u_2(x)\rangle_{\pi_\tau^{-1}(x)}\,dx=
\int_G \langle f_1(g),f_2(g)\rangle_H\,dg
\end{equation}
so that $u\longleftrightarrow f$ is an isometry between the space $L^2(\xi_\tau)$ of
the square integrable sections of $\xi_\tau$ and the space $L^2_\tau(G)$ of the square
integrable functions of type $\tau$.

The left regular action of $G$ on $L^2_\tau(G)$ (also called the
\emph{representation of $G$ induced by $\tau$})
$$
L_h f(g) := f(h^{-1}g)
$$
can be equivalently realized on $L^2(\xi_\tau)$ by
\begin{equation}\label{indotta}
 U_h u(x) = h u(h^{-1}x)\ .
\end{equation}
We have
$$
U_h u(gK) = h u(h^{-1}gK) = h \theta( h^{-1} g, f(h^{-1}g)) =
\theta(g, f(h^{-1}g)) = \theta(g, L_h f(g))
$$
so that, thanks to the uniqueness of the pullback function:
\begin{prop}\label{action-pullback}
If $f$ is the pullback function of the section $u$ then $L_hf$ is
the pullback of the section $U_hu$.
\end{prop}
Let  $T=(T_x)_{x\in \X}$ be
a random section of the homogeneous vector bundle $\xi_\tau$. As, for
fixed $\omega$, $x\mapsto T_x(\omega)$ is a section of $\xi_\tau$, by Proposition
\ref{pullback-s-deterministic} there
exists a unique function $g\mapsto X_g(\omega)$ of type $\tau$ such that
$T_{gK}(\omega)=\theta(g, X_g(\omega))$. We refer to the
random field $X=(X_g)_{g\in G}$ as the \emph{pullback random field
of $T$}. It is a random field on $G$ of type $\tau$, i.e.  $X_{gk}(\omega)=\tau(k^{-1})X_g(\omega)$ for each $\omega$.
Conversely every random field $X$ on $G$ of type $\tau$ uniquely defines a
random section of $\xi_\tau$ whose pullback random field is $X$. It is immediate that
\begin{prop}\label{prop-pull1}
Let $T$ be a random section of $\xi_\tau$.

a) $T$  is  a.s. square integrable if and only if its pullback random field $X$ is a.s. square
integrable.

b) $T$ is second order if and only if its pullback random field $X$ is second order.

c) $T$ is a.s. continuous if and only if its pullback random field $X$ is a.s. continuous.
\end{prop}
Proposition \ref{prop-pull1} introduces the fact that many properties of random sections of
the homogeneous bundle can be stated or investigated through corresponding properties of
their pullbacks, which are just ordinary random fields to whom the results of previous
sections can be applied. A first instance is the following definition.
\begin{definition}\label{definizione di continuita in media quadratica}
The random section $T$ of the homogeneous vector bundle $\xi_\tau$ is
said to be \emph{mean square continuous} if its  pullback  random
field $X$ is mean square continuous, i.e.,
\begin{equation}
\lim_{h\to g} \E [ \| X_h - X_g \|_H^2 ]=0\ .
\end{equation}
\end{definition}
Recalling Definition \ref{invarian}, we state now the definition of
strict-sense invariance.
Let $T$ be an a.s. square integrable random section of
$\xi_\tau$.
For every $g\in G$, the ``rotated'' random section
$T^g$ is defined as
\begin{equation}
T^g_x(\cdot):= g^{-1} T_{gx}(\cdot)
\end{equation}
which is still an a.s. square integrable random section of $\xi_\tau$. For any square integrable
section $u$ of $\xi_\tau$, let
\begin{equation}
T(u):= \int_{\X} \langle T_x, u(x)\rangle_{\pi^{-1}(x)}\,dx\ .
\end{equation}
\begin{definition}\label{isotropia per sezioni aleatorie}
Let $T$ be an a.s. square integrable random section of the homogeneous vector bundle
$\xi_\tau$. It is said to be \emph{(strict-sense) $G$-invariant}
or \emph{isotropic} if and only if
for every choice of square integrable sections
$u_1, u_2, \dots, u_m$ of $\xi_\tau$, the random vectors
\begin{equation}\label{= in legge}
\bigl( T(u_1), \dots, T(u_m) \bigr)\quad
\mbox{and} \quad\bigl( T^g(u_1), \dots, T^g(u_m) \bigr)=\bigl( T( U_g u_1),  \dots, T( U_g u_m) \bigr)
\end{equation}
have the same law for every $g\in G$.
\end{definition}
\begin{prop}\label{pullback-invariant}
Let $T$ be an a.s. square integrable random section of $\xi_\tau$ and let
$X$ be its pullback random field on $G$. Then $X$ is isotropic
if and only if $T$ is an isotropic random section.
\end{prop}
\begin{proof}
Let us denote $X(f) := \int_G \langle X_g, f(g) \rangle_H\,dg$.
Thanks to Proposition \ref{action-pullback} the equality in law (\ref{= in legge}) is equivalent to the requirement that
for every choice of square integrable functions $f_1, f_2, \dots, f_m$ of type $\tau$ (i.e. the pullbacks of corresponding sections of $\xi_\tau$), the random vectors
\begin{equation}\label{pullback invariante}
( X(f_1), \dots, X(f_m) )\quad \mbox{and}\quad( X(L_g f_1), \dots, X(L_g  f_m) )
\end{equation}
have the same law for every $g\in G$.
As $L^2_\tau(G)$ is a closed subspace of $L^2(G)$ and is invariant under the left regular
representation of $G$, every square integrable function $f:G\to H$
can be written as the sum $f^{(1)}+ f^{(2)}$
with $f^{(1)}\in L^2_\tau(G)$, $ f^{(2)}\in L^2_\tau(G)^{\bot}$. As the paths of the random field $X$ are of type $\tau$ we have $X(f^{(2)})=X(L_h f^{(2)})=0$ for every $h\in G$ so that
\begin{equation}
X(f)=X(f^{(1)})\quad  \mbox{and}\quad  X(L_h f) = X(L_h f^{(1)})\ .
\end{equation}
Therefore (\ref{pullback invariante}) implies that, for every choice $f_1, f_2, \dots, f_m$
of square integrable $H$-valued functions on $G$, the random vectors
\begin{equation}
( X(L_g f_1),  \dots, X(L_g f_m) )\quad  \mbox{and} \quad  ( X(f_1),  \dots, X(f_m) )
\end{equation}
have the same law for every $h\in G$ so that the pullback random field $X$ is a strict-sense isotropic random field on $G$.

\end{proof}
As a consequence of Proposition \ref{mean square continuous}
(see also \cite{mp-continuity}) we have
\begin{cor} Every a.s.
square integrable, second order and isotropic random section $T$ of
the homogeneous vector bundle $\xi_\tau$ is mean square
continuous.
\end{cor}

In order to make a comparison with the pullback approach developed above, we briefly recall
the approach to the theory of random fields in vector bundles introduced by Malyarenko in \cite{bib:M}.
The main tool is the
scalar random field associated to the random section $T$ of $\xi=(E,p,B)$.
More precisely, it is the complex-valued random field $T^{sc}$ indexed by the elements $\eta\in E$ given by
\begin{equation}\label{scalar random field}
T^{sc}_{\eta} := \langle \eta, T_{b} \rangle_{p^{-1}(b)}, \; b\in B, \eta\in p^{-1}(b)\ .
\end{equation}
$T^{sc}$ is a scalar random field on $E$ and we can give the definition that $T$ is mean square continuous
if and only if $T^{sc}$  is mean square continuous, i.e., if the map
\begin{equation}
E \ni \eta \mapsto T^{sc}_{\eta}\in L_\C^2(\P)
\end{equation}
is continuous.
Given a  topological group $G$ acting with a continuous action
$(g,x)\mapsto gx, g\in G$ on the base space
$B$, an action of $G$ on $E$ is called associated if its restriction to any fiber $p^{-1}(x)$ is a linear isometry between
$p^{-1}(x)$ and $p^{-1}(gx)$.  In our case of interest, i.e. the homogeneous vector
bundles $\xi_\tau=(G\times_\tau H, \pi_\tau, \X)$, we can consider the action defined in \paref{action} which is obviously associated. We can now define that $T$ is strict sense $G$-invariant w.r.t. the action of $G$ on $B$ if the finite-dimensional distributions of $T^{sc}$ are invariant under the associated action \paref{action}. In the next statement we prove the equivalence of the two approaches.
\begin{prop} The square integrable random section $T$ of the homogeneous bundle $\xi_\tau$ is mean square continuous (i.e. its pullback random field on $G$ is mean square continuous)  if and only if the associated scalar random field $T^{sc}$ is mean square continuous. Moreover if $T$ is a.s. continuous then it is isotropic if and only if the associated scalar random field $T^{sc}$ is $G$-invariant.
\end{prop}
\begin{proof}
Let $X$ be the pullback random field of $T$. Consider the scalar random field on $G\times H$ defined as $X^{sc}_{(g,z)} := \langle z, X_g \rangle_H$. Let us denote $pr$ the projection $G\times H\to G\times_\tau H$: keeping in mind (\ref{prod scalare}) we have
\begin{equation}\label{2=}
T^{sc} \circ pr =X^{sc}\ ,
\end{equation}
i.e.
$$
T^{sc}_{\th(g,z)} (\omega) = X^{sc}_{(g,z)}(\omega)
$$
for every $(g,z)\in G\times H$, $\omega\in \Omega$.
Therefore the map $G\times_\tau H \ni \theta(g,z) \mapsto T^{sc}_{\theta(g,z)}\in L^2_\C(P)$ is continuous if and only if the map $G\times H \ni (g,z) \mapsto X^{sc}_{(g,z)}\in L^2_\C(P)$ is continuous, the projection $pr$ being open and continuous.
Let us  show that the continuity of the latter map is equivalent to the mean square continuity of the pullback random field $X$, which will allow to conclude.
The proof of this equivalence is inspired by the one of a similar statement in \cite{bib:M}, $\S 2.2$.

Actually consider an orthonormal basis $\lbrace v_1, v_2, \dots, v_{\dim\tau} \rbrace$ of $H$, and
denote $X^i=\langle X,v_i\rangle$ the $i$-th component of $X$ w.r.t. the above basis. Assume that the map $G\times H \ni (g,z) \mapsto X^{sc}_{(g,z)}\in L^2_\C(P)$ is continuous, then the random field on $G$
$$
g\mapsto   X^{sc}_{(g,v_i)}=\overline{ X_g ^i}
$$
is continuous for every $i=1, \dots, \dim\tau$.
As $\E[| \overline{X_g^i}-\overline{X_h^i}|^2]=\E [| X_g^i - X_h^i|^2]$,
$$
\lim_{h\to g} \E [\| X_g - X_h \|_H^2] =\lim_{h\to g} \sum_{i=1}^{\dim\tau}
\E [| X_g^i - X_h^i|^2] = 0\ .
$$
Suppose that the pullback random field $X$ is mean square continuous.
Then for each $i=1, \dots, \dim\tau$
$$
\dlines{
0 \le \limsup_{h\to g} \E[ | X_g^i - X_h^i|^2] \le \lim_{h\to g} \E[ \| X_g - X_h \|^2_H ]= 0
}$$
so that the maps $G\ni g\mapsto    X_g ^i \in L^2_\C(\P)$ are continuous.
Therefore
$$\dlines{
\lim_{(h,w) \to (g,z)} \E [| X^{sc}_{(h,w)} - X^{sc}_{(g,z)} |^2] \le 2 \sum_{i=1}^{\dim\tau} \lim_{(h,w) \to (g,z)} \E[|w_i X_h^i - z_i X_g^i |^2] = 0\ ,
}$$
$a_i$ denoting the $i$-th component of $a\in H$.

Assume that $T$ is a.s. continuous and let us show that it is  isotropic if and only if the associated scalar random field $T^{sc}$ is $G$-invariant.
Note first that, by \paref{2=} and $(T^{sc})^h=(X^{sc})^h\circ pr$ for any $h\in G$,  $T^{sc}$ is $G$-invariant if and only if  $X^{sc}$ is $G$-invariant.
Actually if the random fields $X^{sc}$ and its rotated $(X^{sc})^h$
have the same law, then $T^{sc} \mathop{=}^{law} X^{sc}$ and
vice versa.
Now recalling the definition of $X^{sc}$, it is obvious that  $X^{sc}$ is $G$-invariant
if and only if $X$ is isotropic.

\end{proof}
\section{Random sections of the homogeneous line bundles on $\cS^2$}
We now concentrate on the case of the homogeneous line bundles on
$\X=\cS^2$ with $G=SO(3)$ and $K\cong SO(2)$.  For every character $\chi_s$ of $K$, $s\in\Z$, let $\xi_s$ be the corresponding homogeneous vector bundle on $\cS^2$, as
explained in the previous section.
Given the action of $K$ on $SO(3)\times \mathbb{C}$:
$k(g,z)=(gk, \chi_s(k^{-1})z)$, $k\in K$, let $\mE_s:=SO(3)\times_s\C$ be the space of the orbits
$\mE_s=\lbrace \theta(g,z), (g,z)\in G\times \mathbb{C}\rbrace$
where $\theta(g,z) = \lbrace (gk, \chi_s(k^{-1})z); k\in K \rbrace$.
If  $\pi_s:  \mE_s \ni\theta(g,z)\to gK\in \cS^2$,
$\xi_s=(\mE_s, \pi_s, \cS^2)$ is an homogeneous line bundle (each fiber $\pi_s^{-1}(x)$ is isomorphic  to $\C$ as a vector space).

The space $L^2(\xi_s)$ of the square integrable sections of
$\xi_s$ is therefore isomorphic to the space $L^2_s(SO(3))$ of the
square integrable \emph{functions of type $s$}, i.e. such that, for every $g\in G$ and $k \in K$,
\begin{equation}
f(gk)=\chi_s(k^{-1})f(g)=\overline{\chi_s(k)}f(g)\ .
\end{equation}
Let us investigate the Fourier expansion of a function of type $s$.

\begin{prop}\label{infinite-linear}
Every function of type $s$ is an infinite linear combination of the
functions appearing in the $(-s)$-columns of Wigner's $D$ matrices
$D^\ell$, $\ell \ge |s|$. In particular functions of type $s$ and type
$s'$ are orthogonal if $s\not=s'$.

\end{prop}
\begin{proof}
For every $\ell \ge |s|$, let $\widehat f(\ell)$ be as in (\ref{coefficiente ellesimo}). We have,
for every $k\in K$,
\begin{equation}\label{leading}
\begin{array}{c}
\displaystyle\widehat f(\ell)
 = \sqrt{2\ell+1}\,\int_{SO(3)}f(g)D^\ell(g^{-1})\,dg =
 \sqrt{2\ell+1}\,\chi_s(k)\int_{SO(3)} f(gk) D^\ell(g^{-1})\,dg=\\
\displaystyle=\sqrt{2\ell+1}\,\chi_s(k)\int_{SO(3)} f(g) D^\ell(kg^{-1})\,dg
=\sqrt{2\ell+1}\,\chi_s(k)D^\ell(k)\int_{SO(3)} f(g)
D^\ell(g^{-1})\,dg=\\
 =\chi_s(k)D^\ell(k)\widehat f(\ell) \ ,\\
\end{array}
\end{equation}
i.e. the image of $\widehat f(\ell)$ is contained in the subspace
$H_{\ell}^{(-s)} \subset H_\ell$ of the vectors such that $D^\ell(k)v=
\chi_{-s}(k)v$ for every $k\in K$. In particular $\widehat f(\ell) \ne 0$ only if
$\ell \ge |s|$, as for every $\ell$ the restriction to $K$ of the representation
$D^\ell$
is unitarily equivalent to the direct sum of the representations $\chi_m$,
$m=-\ell, \dots, \ell$ as recalled at the end of \S2.

Let $\ell \ge |s|$ and $v_{-\ell}, v_{-\ell + 1}, \dots, v_{\ell}$ be the orthonormal basis of $H_\ell$
as in (\ref{restrizione}), in other words $v_m$ spans $H_{\ell}^{m}$, the
one-dimensional subspace of $H_\ell$
formed by the vectors that transform under the action of $K$ according to the
representation $\chi_{m}$.
It is immediate that
\begin{align}
\widehat f(\ell)_{i,j} = \langle \widehat f(\ell)  v_j, v_i \rangle=0\ ,
\end{align}
unless $i=-s$. Thus the Fourier coefficients of $f$ vanish but those corresponding to
the column $(-s)$ of the matrix representation $D^\ell$ and the Peter-Weyl expansion  (\ref{PW SO(3)}) of $f$ becomes, in  $L^2(SO(3))$,
\begin{equation}\label{sviluppo per una funzione di tipo s}
f=\sum_{\ell \ge |s|} \sqrt{2\ell + 1} \sum_{m=-\ell}^{\ell}
\widehat f(\ell)_{-s,m} D^\ell_{m,-s}\ .
\end{equation}
\end{proof}
We introduced the spherical harmonics in (\ref{armoniche sferiche1}) from the entries $D^{\ell}_{m,0}$ of the central columns of Wigner's $D$ matrices.
Analogously to the case of $s=0$, for any $s\in \Z$ we define for $\ell \ge |s|, m=-\ell, \dots, \ell$
\begin{equation}\label{armoniche di spin}
_{-s} Y_{\ell,m} (x) := \theta \Bigl( g, \sqrt{\frac{2\ell +1}{4\pi}}\, \overline{D^\ell_{m,s}(g)} \Bigr)\ , \qquad x=gK\in \cS^2\ .
\end{equation}
$ _{-s} Y_{\ell,m}$
is a section of $\xi_s$ whose pullback function (up to a factor)
is $g\mapsto D^{\ell}_{-m,-s}(g)$ (recall the relation
$\overline{D^\ell_{m,s}(g)} = (-1)^{m-s} D^{\ell}_{-m,-s}(g)$, see \cite{dogiocam} p.~55 e.g.).
Therefore thanks to Proposition \ref{infinite-linear} the sections
$_{-s} Y_{\ell,m},\, \ell \ge |s|, m=-\ell, \dots, \ell$, form an
\emph{orthonormal} basis of $L^2(\xi_s)$. Actually recalling
(\ref{prod scalare}) and (considering the total mass equal to $4\pi$ on the sphere and to $1$ on $SO(3)$)
$$
\int_{\cS^2}\, _{-s} Y_{\ell,m}\, \overline{_{-s}Y_{\ell',m'}}\,dx =
4\pi \int_{SO(3)} \sqrt{\frac{2\ell +1}{4\pi}}\, \overline{D^\ell_{m,s}(g)}\,
\sqrt{\frac{2\ell' +1}{4\pi}}\,D^{\ell'}_{m',s}(g)\,dg = \delta_{\ell'}^{\ell}\delta_{m'}^{m}\ .
$$
The sections  $_{-s} Y_{\ell,m},\, \ell \ge |s|, m=-\ell, \dots, \ell$ in
(\ref{armoniche di spin}) are called  \emph{spin $-s$ spherical harmonics}.
Recall that the spaces $L^2_s(SO(3))$ and $L^2(\xi_s)$ are isometric through
the identification $u \longleftrightarrow f$ between a section $u$ and its
pullback $f$ and the definition of the scalar product on $L^2(\xi_s)$ in
(\ref{prod_scalare}).
Proposition (\ref{infinite-linear}) can be otherwise stated as

\noindent \emph{Every square integrable section $u$ of the homogeneous line
bundle $\xi_s=(\mE_s, \pi_s, \cS^2)$ admits a Fourier expansion in terms of spin
$-s$ spherical harmonics of the form
\begin{equation}
u(x) = \sum_{\ell \ge |s|} \sum_{m=-\ell}^{\ell} u_{\ell,m}\, _{-s} Y_{\ell,m}(x)\ ,
\end{equation}
where $u_{\ell,m} := \langle u,\, _{-s} Y_{\ell,m} \rangle_2$, the above series converging in $L^2(\xi_s)$.}

In particular we have the relation
$$
\dlines{
 u_{\ell,m} = \int_{\cS^2} u(x)\,   _{-s} Y_{\ell,m}(x) \,dx =
4\pi \int_{SO(3)} f(g) \sqrt{ \frac{2\ell +1}{4\pi}} D^\ell_{m,s}(g)\,dg =\cr
(-1)^{s-m} \sqrt{4\pi (2\ell+1)} \int_{SO(3)} f(g)  \overline{  D^\ell_{-m,-s}(g)}\,dg = (-1)^{s-m} \sqrt{4\pi}\, \widehat f(\ell)_{-s,-m}\ . }
$$
\begin{definition}
Let $s\in \Z$. A square integrable function $f$ on $SO(3)$ is said
to be \emph{bi-$s$-associated} if for every $g\in SO(3), k_1, k_2 \in K$,
\begin{equation}
f(k_1 g k_2) = \chi_s (k_1) f(g) \chi_s(k_2)\ .
\end{equation}
\end{definition}
Of course for $s=0$ {bi-$0$-associated} is equivalent to bi-$K$-invariant.
We are particularly interested in bi-$s$-associated functions
as explained in the remark below.
\begin{remark}\label{associate-bi} \rm Let $X$ be an isotropic random field of type $s$ on
$SO(3)$. Then its associate positive definite function $\phi$ is
bi-$(-s)$-associated. Actually, assuming for simplicity that $X$ is
centered, as $\phi(g)=\E[X_g\overline{X_e}]$, we have, using
invariance on $k_1$ and type $s$ property on $k_2$,
$$
\phi(k_1gk_2)=\E[X_{k_1gk_2}\overline{X_e}]=\E[X_{gk_2}\overline{X_{k_1^{-1}}}]=
\chi_s(k_1^{-1})\E[X_g\overline{X_e}]\chi_s(k_2^{-1})=
\chi_{-s}(k_1)\phi(g)\chi_{-s}(k_2)\ .
$$
\end{remark}
\qed

Let us investigate the Fourier expansion of a bi-$s$-associated
function $f$: note first that a bi-$s$-associated function is also a
function of type $(-s)$, so that $\widehat f(\ell) \ne 0$
only if $\ell \ge |s|$ as above and all its rows vanish but for
the $s$-th. A repetition of the computation
leading to \paref{leading} gives easily that
$$
\widehat f(\ell)=\chi_{-s}(k)\widehat f(\ell)D^\ell(k)
$$
so that the only non-vanishing entry of the matrix $\widehat f(\ell)$ is
the $(s,s)$-th.

Therefore the Fourier expansion of a bi-$s$-associated function $\phi$ is
\begin{equation}\label{sviluppo per una funzione bi-s-associata}
f= \sum_{\ell \ge |s|} \sqrt{2\ell + 1}\, \alpha_\ell
D^\ell_{s,s}\ ,
\end{equation}
where we have set $\alpha_\ell=\widehat f(\ell)_{s,s}$.

Now let $T$ be an a.s. square integrable random section of the line
bundle $\xi_s$ and $X$ its pullback random field. Recalling that $X$
is a random function of type $s$ and its sample paths are a.s.
square integrable, we easily obtain the stochastic Fourier
expansion of $X$ applying (\ref{sviluppo per una funzione di tipo s})
to the functions $g\mapsto X_g(\omega)$. Actually define, for every $\ell \ge |s|$, the random operator
\begin{equation}
\widehat X(\ell)= \sqrt{2\ell + 1}\int_{SO(3)} X_g D^\ell(g^{-1})\,dg\ .
\end{equation}
The basis of $H_\ell$ being fixed as above and recalling (\ref{sviluppo per una funzione di tipo s}), we obtain, a.s. in $L^2(SO(3))$,
\begin{equation}\label{sviluppo di Fourier per X}
X_g =\sum_{\ell \ge |s|} \sqrt{2\ell + 1}
\sum_{m=-\ell}^{\ell} \widehat X(\ell)_{-s,m} D^\ell_{m,-s}(g)\ .
\end{equation}

If $T$ is isotropic, then by Definition
\ref{isotropia per sezioni aleatorie} its pullback
random field $X$ is also isotropic in the sense of Definition
\ref{invarian}. The following is a consequence of well known general properties of the random
coefficients of invariant random fields (see \cite{trap} Theorem 3.2 or \cite{bib:M} Theorem 2).
\begin{prop}\label{teorema di struttura}
Let $s\in \Z$ and $\xi_s=(\mE_s, \pi_s, \cS^2)$ be the homogeneous
line bundle on $\cS^2$ induced by the $s$-th linear character $\chi_s$ of $SO(2)$.
Let $T$ be a random section
of $\xi_s$ and $X$ its pullback random field. If $T$ is second order and strict-sense isotropic, then the Fourier coefficients $X(\ell)_{-s,m}$ of $X$
in its stochastic expansion \paref{sviluppo di Fourier per X}
are pairwise orthogonal and the variance, $c_\ell$, of $\widehat X(\ell)_{-s,m}$ does not depend on $m$.
Moreover $\E [ \widehat X(\ell)_{-s,m} ]=0$ unless $\ell=0, s=0$.
\end{prop}

For the random field $X$ of Proposition \ref{teorema di struttura} we have
immediately
\begin{equation}\label{conv}
\E[|X_g|^2]=\sum_{\ell \ge |s|} (2\ell + 1) c_\ell < +\infty\ .
\end{equation}
The convergence of the series above is also a consequence of
Theorem \ref{gangolli}, as the positive definite function $\phi$ associated to
$X$ is given by
$$
\dlines{
\phi(g)=\E[X_g\overline{X_e}]=\sum_{\ell \ge |s|} (2\ell + 1)c_\ell
\sum_{m=-\ell}^{\ell} D^\ell_{m,-s}(g)\overline{D^\ell_{m,-s}(e)}=\cr
=\sum_{\ell \ge |s|} (2\ell + 1)c_\ell
\sum_{m=-\ell}^{\ell} D^\ell_{m,-s}(g)D^\ell_{-s,m}(e)=
\sum_{\ell \ge |s|} (2\ell + 1)c_\ell D^\ell_{-s,-s}(g)\ . \cr
}
$$
\begin{remark} \rm
Let $X$ be a type $s$ random field on $SO(3)$ with $s\not=0$. Then the
relation $X_{gk}=\chi_s(k^{-1})X_{g}$ implies that $X$ cannot be real
(unless it is vanishing). If in addition it was Gaussian,
then, the identity in law between $X_g$ and $X_{gk}=\chi_s(k^{-1})X_{g}$
would imply that, for every $g\in G$, $X_g$ is a complex Gaussian r.v.
\end{remark}

\section{Construction of Gaussian isotropic spin random fields}

We now give an extension of the construction of \S\ref{sec4}
and prove that every complex Gaussian random section of a homogeneous
line bundle on $\cS^2$ can be obtained in this way, a result much
similar to Theorem \ref{real-general}.
Let $s\in \Z$ and let $\xi_s$ be the homogeneous line bundle associated to
the representation $\chi_s$.

Let $(X_n)_n$ be a sequence of i.i.d. standard Gaussian r.v.'s on some probability space $(\Omega, \F, \mathbb{P})$, and
$\mathscr{H}\subset L^2(\Omega,\F,\P)$ the \emph{complex} Hilbert space
generated by $(X_n)_n$.
Let $(e_n)_n$ be an orthonormal basis of $L^2(SO(3))$ and
define an isometry $S$ between $L^2(SO(3))$ and
$\mathscr{H}$ by
$$
L^2(SO(3))\ni \sum_k \alpha_k e_k\enspace \to\enspace \sum_k \alpha_k X_k \in \mathscr{H}\ .
$$
Let $f\in L^2(SO(3))$, we define a random field $X^f$ on $SO(3)$ by
\begin{align}\label{spin-def}
X^f_g=S(L_gf)\ ,
\end{align}
$L$ denoting as usual the left regular representation.
\begin{prop}
If $f$ is a square integrable bi-$s$-associated function
on $SO(3)$, then $X^f$ defined in \paref{spin-def} is a second order,
square integrable Gaussian isotropic random field of type $s$. Moreover it is
complex Gaussian.
\end{prop}
\begin{proof}
It is immediate that $X^f$ is second order as
$\E[|X^f_g|^2]=\Vert L_gf\Vert^2_2=\Vert f\Vert^2_2$. It
 is of type $s$ as for every $g\in SO(3)$ and $k\in K$,
$$
X^f_{gk}=S(L_{gk}f)=\chi_s(k^{-1})S(L_gf)=\chi_s(k^{-1})X^f_g\ .
$$
Let us prove strict-sense invariance. Actually, $S$ being an isometry,
for every $h\in SO(3)$
$$
\dlines{
\E[ X^f_{hg}\overline{X^f_{hg'}}] = \E[ S(L_{hg}f)\overline{S(L_{hg'}f)}]
= \langle L_{hg}f, L_{hg'}f \rangle_2 =
\langle L_{g}f, L_{g'}f \rangle_2
= \E[X^f_{g}\overline{X^f_{g'}}]\ .
}$$
Therefore the random fields $X^f$ and its rotated $(X^f)^h$
have the same covariance
kernel. Let us prove that they also have the same relation function.
Actually we have, for every $g,g'\in SO(3)$,
\begin{equation}\label{complex-spin1}
\E[X^f_{g}X^f_{g'}]=\E[ S(L_{hg}f){S(L_{hg'}f)}]
= \langle L_{hg}f, \overline {L_{hg'}f} \rangle_2 =\langle L_{g}f, \overline{L_{g'}f} \rangle_2=0
\end{equation}
as the function $\overline{L_{g'}f}$ is bi-$(-s)$-associated and therefore of
type $s$ and orthogonal to $L_{g}f$ which is of type $-s$ (orthogonality of functions
of type $s$ and $-s$ is a consequence of Proposition \ref{infinite-linear}).

In order to prove that $X^f$ is complex Gaussian we must show that
for every $\psi\in L^2(SO(3))$, the r.v.
$$
Z=\int_{SO(3)}X^f_g\psi(g)\, dg
$$
is complex Gaussian. As $Z$ is Gaussian by construction we must just prove that
$\E[Z^2]=0$. But as, thanks to \paref{complex-spin1}, $\E[X^f_{g}X^f_{g'}]=0$
$$
\E[Z^2]=\E\Bigl[\int_{SO(3)}\int_{SO(3)}X^f_gX^f_h\psi(g)\psi(h)\,dg dh\Bigr]=
\int_{SO(3)}\int_{SO(3)}\E[X^f_gX^f_h]\psi(g)\psi(h)\,dg dh=0\ .
$$
\end{proof}

Let us investigate the stochastic Fourier expansion of
$X^f$.
Let us consider first the random field $X^\ell$ associated to
$f=D^\ell_{s,s}$. Recall first that the r.v. $Z=S(D^\ell_{s,s})$
has variance $\E[|Z|^2]=\Vert D^\ell_{s,s}\Vert_2^2=(2\ell+1)^{-1}$ and that
$\overline{D^\ell_{m,s}}=(-1)^{m-s}D^\ell_{-m,-s}$. Therefore
$$
\dlines{
X^\ell_g =S(L_g D^\ell_{s,s})=
\sum_{m=-\ell}^{\ell} S(D^\ell_{m,s}) D^\ell_{s,m}(g^{-1})=\cr
=\sum_{m=-\ell}^{\ell} S(D^\ell_{m,s})
\overline{D^\ell_{m,s}(g)}=
\sum_{m=-\ell}^{\ell} S(D^\ell_{m,s}) (-1)^{m-s}D^\ell_{-m,-s}(g)\ .\cr
}
$$
Therefore the r.v.'s
$$
a_{\ell,m}=\sqrt{2\ell+1}\, S(D^\ell_{m,s})(-1)^{m-s}
$$
are complex Gaussian, independent and with variance $\E[|a_{\ell,m}|^2]=1$ and
we have the expansion
\begin{equation}
X^\ell_g=
\frac{1}{\sqrt{2\ell+1}}\sum_{m=-\ell}^{\ell} a_{\ell, m} D^\ell_{m,-s}(g)\ .
\end{equation}
Note that the coefficients
$a_{\ell m}$ are independent complex Gaussian  r.v.'s.
This is a difference with respect to the case $s=0$,
where in the case of a real random field, the coefficients $a_{\ell,m}$ and $a_{\ell,-m}$ were not independent. Recall that random fields of type $s\ne 0$ on $SO(3)$ cannot be real.

In general, for a square integrable bi-$s$-associated function $f$
\begin{equation}\label{fs}
f=\sum_{\ell \ge |s|} \sqrt{2\ell + 1}\, \alpha_\ell D^\ell_{s,s}
\end{equation}
with
$$
\Vert f\Vert_2^2=\sum_{\ell \ge |s|} |\alpha_\ell|^2 < +\infty\ ,
$$
the Gaussian random field $X^f$ has the expansion
\begin{align}\label{sviluppo per X}
X^f_g&=\sum_{\ell \ge |s|} \alpha_\ell
\sum_{m=\ell}^{\ell} a_{\ell, m} D^\ell_{m,-s}(g)\ ,
\end{align}
where $(a_{\ell,m} )_{\ell,m}$ are independent  complex Gaussian
r.v.'s with zero mean and unit variance.

The associated positive definite function of $X^f$,
$\phi^f(g):=\E[X^f_g \overline{X^f_e}]$ is bi-$(-s)$-associated (Remark
\ref{associate-bi}) and
continuous (Theorem \ref{gangolli}) and, by \paref{convolution for phi},
is related to $f$ by
$$
\phi^f=f \ast \breve f (g^{-1})\ .
$$
This allows to derive its Fourier expansion:
$$
\dlines{
\phi^f(g)=f \ast \breve f (g^{-1})=
\int_{SO(3)} f(h) \overline{f(g h)}\,dh
=\sum_{\ell, \ell' \ge |s|} \sqrt{2\ell + 1}\,\sqrt{2\ell'+1}\, \alpha_\ell \overline{\alpha_{\ell'}}
\int_{SO(3)} D^\ell_{s,s}(h) \overline{D^{\ell'}_{s,s}(g h)}\,dh=\cr
=\sum_{\ell, \ell' \ge |s|} \sqrt{2\ell + 1}\,\sqrt{2\ell'+1}\, \alpha_\ell \overline{\alpha_{\ell'}}
\sum_{j=-\ell}^{\ell} \underbrace{\Bigl (\int_{SO(3)} D^\ell_{s,s}(h)
\overline{D^{\ell'}_{j,s}(h)}\,dh\, \Bigr)}_{= \frac 1 {2\ell +1}
\delta_{\ell,\ell'} \delta_{s,j}} \overline{D^{\ell}_{s,j}(g)}=\cr
=\sum_{\ell \ge |s|} |\alpha_\ell |^2 D^\ell_{-s,-s}(g)\ .
}
$$
Note that in accordance with Theorem \ref{gangolli}, as
$|D^\ell_{-s,-s}(g)|\le D^\ell_{-s,-s}(e)=1$, the above series converges uniformly.

Conversely, it is immediate that, given a continuous positive definite
bi-$(-s)$-associated function $\phi$, whose expansion is
$$
\phi^f(g)=\sum_{\ell \ge |s|} |\alpha_\ell |^2 D^\ell_{-s,-s}(g)\ ,
$$
by choosing
$$
f(g)=\sum_{\ell \ge |s|} \sqrt{2\ell+1}\,\beta_\ell D^\ell_{-s,-s}(g)
$$
with $|\beta_\ell|=\sqrt{\alpha_\ell}$,
there exist a square integrable bi-$s$-associated function $f$ as in \paref{fs}
such that $\phi(g)=f*\breve f(g^{-1})$. Therefore, for every random field
$X$ of type
$s$ on $SO(3)$ there exists a square integrable bi-$s$-associated
function $f$ such that $X$ and $X^f$ coincide in law. Such a function $f$ is not
unique.

From $X^f$ we can define a random section $T^f$ of the homogeneous
line bundle $\xi_s$ by
\begin{equation}
T^f_{x} := \theta(g, X^f_g)\ ,
\end{equation}
where $x=gK\in \cS^2$.
Now, as for the case $s=0$ that was treated in \S\ref{sec4}, it is natural to ask whether every Gaussian isotropic section of $\xi_s$ can be obtained in this way.
\begin{theorem}
Let $s\in \Z\setminus \{0\}$.
For every square integrable, isotropic, (complex) Gaussian
random section  $T$ of the homogeneous  $s$-spin line bundle $\xi_s$, there exists a square integrable and bi-$s$-associated function $f$ on $SO(3)$
such that
\begin{equation}\label{=}
T^f\enspace \mathop{=}^{law}\enspace T\ .
\end{equation}
Such a function $f$ is not unique.
\end{theorem}
\begin{proof}
Let $X$ be the pullback random field (of type $s$) of $T$. $X$ is of course
mean square continuous. Let $R$ be its covariance kernel. The function
$\phi(g):=R(g,e)$  is  continuous,
positive definite and bi-$(-s)$-associated, therefore has the expansion
\begin{equation}\label{sviluppo per fi 2}
\phi=\sum_{\ell \ge |s|} \sqrt{2\ell + 1}\,\beta_\ell D^\ell_{-s,-s}\ ,
\end{equation}
where
$\beta_\ell=\sqrt{2\ell + 1}\,\int_{SO(3)} \phi(g)
\overline{D^\ell_{-s,-s}(g)}\,dg \ge 0$. Furthermore,
by Theorem \ref{gangolli}, the series
in (\ref{sviluppo per fi 2}) converges uniformly, i.e.
$$
\sum_{\ell \ge |s|} \sqrt{2\ell + 1}\,\beta_\ell < +\infty\ .
$$
Now set $f:=\sum_{\ell \ge |s|} (2\ell +1)\sqrt{\beta_\ell} D^\ell_{s,s}$.
Actually, $f\in L^2_s(SO(3))$ as $\| f\|^2_{L^2(SO(3))} = \sum_{\ell \ge |s|} (2\ell +1)\beta_\ell < +\infty$ so that it is bi-$s$-associated.

Note that every function $f$ of the form
$f=\sum_{\ell \ge |s|} (2\ell +1)\alpha_\ell D^\ell_{s,s}$ where
$\alpha_\ell$ is such that $\alpha_\ell^2=\beta_\ell$ satisfies
(\ref{=}) (and clearly every function $f$ such that $\phi(g)=f*\breve f(g^{-1})$ is of this form).

\end{proof}
\section{The connection with classical spin theory}
There are different approaches to the theory of random sections of homogeneous line bundles
 on $\cS^2$ (see \cite{marinuccigeller}, \cite{bib:LS}, \cite{bib:M}, \cite{bib:NP}  e.g.). In this section
we compare them, taking into account, besides the one outlined
 in \S 6,  the classical Newman and Penrose  spin theory (\cite{bib:NP})
later formulated in a more mathematical framework by Geller and Marinucci
(\cite{marinuccigeller}).

Let us first recall some basic notions about vector bundles. From now on $s\in \Z$. We shall state them concerning
the complex line bundle $\xi_s=(\mE_s, \pi_s, \cS^2)$ even if they can be immediately extended
to more general situations. An atlas of $\xi_s$ (see \cite{Husemoller} e.g.) can be defined as follows.
Let $U\subset \cS^2$ be an open set and $\Psi$ a diffeomorphism between $U$ and an open set of $\R^2$. A chart $\Phi$ of $\xi_s$ over $U$ is an isomorphism
\begin{equation}\label{def chart}
\Phi: \pi^{-1}_s(U)\goto \Psi(U)\times \C\ ,
\end{equation}
whose restriction to every fiber $\pi_s^{-1}(x)$ is a linear isomorphism $\leftrightarrow\C$. An atlas of $\xi_s$ is a family $( U_j, \Phi_j)_{j\in J}$ such that $\Phi_j$ is a chart of $\xi_s$ over $U_j$
and the family $(U_j)_{j\in J}$ covers $\cS^2$.

Given an atlas $( U_j, \Phi_j)_{j\in J}$, For each pair $i,j\in J$ there exists a unique map
(see \cite{Husemoller} Prop. 2.2) $\lambda_{i,j}: U_i\cap U_j \goto \C\setminus 0$
such that for $x\in U_i\cap U_j, z\in \C$,
\begin{equation}
\Phi_i^{-1}(\Psi_i(x),z)=\Phi_j^{-1}(\Psi_j(x),\lambda_{i,j}(x)z)\ .
\end{equation}
The map $\lambda_{i,j}$ is called the \emph{transition function} from the chart
$(U_j,\Phi_j)$ to the chart $(U_i,\Phi_i)$.
Transition functions satisfy the cocycle conditions, i.e.
for every $i,j,l\in J$
\begin{equation}
\begin{array}{l}\label{cociclo}
\nonumber
\lambda_{j,j} = 1\ \ \ \  \qquad  \text{on}\ \ \ U_j\ ,\cr
\lambda_{j,i} = \lambda_{i,j}^{-1}\ \ \ \   \quad \text{on}\ \  \ U_i\cap U_j\ ,\cr
\nonumber
\lambda_{l,i}\lambda_{i,j}=\lambda_{l,j}\ \ \  \text{on}\ \  \ U_i\cap U_j\cap U_l\ .
\end{array}
\end{equation}
Recall that we denote $K\cong SO(2)$ the isotropy group of the north pole as in \S 6, \S7,  so that $\cS^2\cong SO(3)/K$.
We show now that an atlas of the line bundle $\xi_s$ is given as soon as we specify

a) an atlas $(U_j,\Psi_j)_{j\in J}$ of the manifold $\cS^2$,

b) for every $j\in J$ a family $(g_x^j)_{x\in U_j}$ of representative
elements $g_x^j\in G$ with $g_x^jK=x$.

\noindent More precisely,  let $(g_x^j)_{x\in U_j}$ be as in b) such that $x\mapsto g^j_x$ is smooth for each $j\in J$.
Let $\eta \in  \pi^{-1}_s(U_j)\subset \mE_s$ and $x:=\pi_s(\eta)\in U_j$,
therefore $\eta=\theta(g^j_x,z)$, for a unique $z\in \C$.
Define the chart $\Phi_j$ of $\xi_s$ over $U_j$ as
\begin{equation}\label{triv}
\Phi_j(\eta)=  (\Psi_j(x), z)\ .
\end{equation}
Transition functions of this atlas are easily determined.
If $\eta \in \xi_s$ is such that $x=\pi_s(\eta)\in U_i\cap U_j$,
then $\Phi_j(\eta)=(\Psi_j(x), z_j)$, $\Phi_i(\eta)=(\Psi_i(x), z_i)$.
As $g_x^iK=g^j_xK$, there exists a unique $k=k_{i,j}(x)\in K$ such that
 $g^j_x=g^i_xk$, so that
$\eta=\theta(g^i_x,z_i)=\theta(g^j_x, z_j)=\theta(g^i_xk, z_j)=\theta(g^i_x, \chi_s(k)z_j)$ which implies $z_i=\chi_s(k)z_j$.
Therefore
\begin{equation}\label{transizione}
\lambda_{i,j}(x)=\chi_s(k)\ .
\end{equation}
\bigskip

\noindent The spin $s$ concept was introduced by Newman and Penrose in \cite{bib:NP}:
\emph{a quantity $u$ defined on $\cS^2$ has spin weight $s$  if, whenever a tangent vector $\rho$
 at any point $x$ on the sphere transforms under coordinate change  by
$\rho'=e^{i \psi} \rho$, then the quantity at this point $x$ transforms
by  $u'=e^{is\psi} u$}. Recently, Geller and Marinucci in \cite{marinuccigeller}
have put this notion in a more mathematical framework modeling such a $u$
as a section of a complex line bundle on $\cS^2$ and they describe this line bundle by giving charts and fixing transition functions to express the transformation laws under
changes of coordinates.

More precisely,
they  define an atlas of $\cS^2$ as follows. They consider the open
covering $(U_R)_{R\in SO(3)}$ of $\cS^2$  given by
\begin{equation}\label{charts}
U_e := \cS^2 \setminus \lbrace x_0, x_1 \rbrace \qquad\text{and}\qquad U_R:=R U_e\ ,
\end{equation}
where $x_0=$the north pole (as usual), $x_1=$the south pole. On $U_e$ they consider the usual spherical coordinates $(\vartheta, \varphi)$, $\vartheta=$colatitude, $\varphi=$longitude
 and on any  $U_R$ the ``rotated'' coordinates $(\vartheta_R, \varphi_R)$  in such a way that $x$ in $U_e$ and $Rx$ in $U_R$ have the same coordinates.

The transition functions are defined as follows.
For each $x\in U_R$, let $\rho_R(x)$ denote the unit  tangent vector at $x$, tangent to the circle $\vartheta_R = const$ and
pointing to the direction of increasing $\varphi_R$. If
$x\in U_{R_1}\cap U_{R_2}$,  let $\psi_{R_2,R_1}(x)$ denote the (oriented) angle from
$\rho_{R_1}(x)$ to $\rho_{R_2}(x)$.
They prove that the quantity
\begin{equation}\label{triv-mg}
e^{is\psi_{R_2,R_1}(x)}
\end{equation}
satisfies the cocycle relations \paref{cociclo} so that this defines a unique (up to isomorphism)
structure of complex line bundle on $\cS^2$ having \paref{triv-mg} as transition
functions at $x$  (see \cite{Husemoller} Th. 3.2).

We shall prove that
this spin $s$ line bundle is the same as the
homogeneous line bundle $\xi_{-s}=(\mE_{-s}, \pi_{-s}, \cS^2)$.
To this aim we have just to check that, for a suitable choice of the atlas
$(U_R, \Phi_R)_{R\in SO(3)}$ of $\xi_{-s}$ of the type described in a), b) above,
the transition functions \paref{transizione}  and \paref{triv-mg} are the same.
Essentially we have to determine the family  $(g^R_x)_{R\in SO(3), x\in U_R}$ as in b).

Recall first that every rotation $R\in SO(3)$ can be realized as a composition of three rotations:
(i) a rotation by an angle $\gamma_R$ around the z axis, (ii) a rotation by an angle
$\beta_R$ around the y axis and (iii) a rotation by an angle $\alpha_R$
around the z axis (the so called z-y-z convention), ($\alpha_R$, $\beta_R$, $\gamma_R$) are the {\it Euler angles} of $R$.
Therefore the rotation $R$ acts on the north pole $x_0$
of $\cS^2$ as mapping $x_0$ to the new location on $\cS^2$
whose spherical coordinates are $(\beta_R, \alpha_R)$ after rotating  the tangent plane at $x_0$
by an angle $\gamma_R$. In each coset $\cS^2\ni x=gK$ let us choose the element $g_x\in SO(3)$ as the rotation such that $g_xx_0=x$ and having its third Euler angle $\gamma_{g_x}$ equal to $0$. Of course
if $x\ne x_0,x_1$, such $g_x$ is unique.

Consider the atlas $(U_R, \Psi_R)_{R\in SO(3)}$ of $\cS^2$ defined as follows.
Set the charts as
\begin{align}
&\Psi_e(x) := (\beta_{g_x}, \alpha_{g_x})\ , \qquad x\in U_e\ ,\\
&\Psi_R(x) := \Psi_e (R^{-1}x)\ , \qquad x\in U_R\ .
\end{align}
Note that for each $R$, $\Psi_R(x)$ coincides with the ``rotated'' coordinates $(\vartheta_R, \varphi_R)$ of $x$.
Let us choose now the family $(g^R_x)_{x\in U_R, R\in SO(3)}$.
For $x\in U_e$ choose $g^e_x:=g_x$
 and for $x\in U_R$
\begin{equation}
g^R_x:=Rg_{R^{-1}x}\ .
\end{equation}
Therefore the corresponding atlas
 $( U_R, \Phi_R)_{R\in SO(3)}$ of $\xi_s$ is given, for $\eta\in \pi^{-1}_s(U_R)$, by
 \begin{equation}
 \Phi_R(\eta)=( \Psi_R(x), z)\ ,
 \end{equation}
 where $x:=\pi_s(\eta)\in U_R$ and $z$ is such that $\eta=\theta(g_x^R,z)$.
Moreover for $R_1, R_2\in SO(3)$, $x\in U_{R_1}\cap U_{R_2}$ we have
\begin{equation}\label{eq k}
k_{R_2,R_1}(x)= (g_{R_2^{-1}x})^{-1} R_2^{-1}R_1 g_{R_1^{-1}x}
\end{equation}
and the transition function
from the chart $(U_{R_1}, \Phi_{R_1})$ to the chart $(U_{R_2}, \Phi_{R_2})$ at $x$ is given by \paref{transizione}
\begin{equation}\label{fnz transizione}
\lambda^{(-s)}_{R_2, R_1}(x):=\chi_s(k)\ .
\end{equation}
From now on let us denote $\omega_{R_2,R_1}(x)$ the rotation angle of $k_{R_2,R_1}(x)$.
Note that, with this choice of the family $(g^R_x)_{x\in U_R, R\in SO(3)}$,
 $\omega_{R_2,R_1}(x)$ is the third Euler angle of the rotation $R_2^{-1}R_1g_{R_1^{-1}x}$.

\begin{remark}\label{particular} \rm Note that we have
$$
R^{-1}g_x=g_{R^{-1}x}\ ,
$$
i.e. $g^R_x=g_x$, in any of the following two situations

a) $R$ is a rotation around the north-south axis (i.e. not changing the latitude of the points of $\cS^2$).

b) The rotation axis of $R$ is orthogonal to the plane $[x_0,x]$ (i.e. changes the colatitude of $x$ leaving its longitude unchanged).

Note that if each of the rotations $R_1,R_2$ are of type a) or of type b), then
$$
k_{R_2,R_1}(x)=g^{-1}_{R_2^{-1}x}R_2^{-1}R_1g_{R_1^{-1}x}=(R_2g_{R_2^{-1}x})^{-1}R_1g_{R_1^{-1}x}=
g_x^{-1}g_x= \mbox{the identity}
$$
and in this case the rotation angle of $k_{R_2,R_1}(x)$ coincides with the angle $-\psi_{R_2,R_1}(x)$, as neither $R_1$ nor $R_2$ change the orientation of the tangent plane at $x$.

Another situation in which the rotation $k$ can be easily computed appears when $R_1$ is the
identity and $R_2$ is a rotation of an angle $\gamma$ around an axis passing through $x$.
Actually
\begin{equation}\label{rot}
k_{R_2,e}(x)=g_x^{-1}R_2^{-1}g_x
\end{equation}
which, by conjugation, turns out to be a rotation of the angle $-\gamma$ around the north-south axis. In this case also it is immediate that the rotation angle $\omega_{R_2,R_1}(x)$ coincides with $-\psi_{R_2,R_1}(x)$.

\qed
\end{remark}
The following relations will be useful in the sequel, setting $y_1=R_1^{-1}x$, $y_2=R_2^{-1}x$,
\begin{align}
k_{R_2,R_1}(x)&=g^{-1}_{R_2^{-1}x}R_2^{-1}R_1g_{R_1^{-1}x}=g^{-1}_{R_2^{-1}R_1y_1}R_2^{-1}R_1g_{y_1}=
k_{R_1^{-1}R_2,e}(R_1^{-1}x)\label{al1}\ ,\cr
k_{R_2,R_1}(x)&=g^{-1}_{R_2^{-1}x}R_2^{-1}R_1g_{R_1^{-1}x}=g^{-1}_{y_2}R_2^{-1}R_1g_{R_1^{-1}R_2y_2}=
k_{e,R_2^{-1}R_1}(R_2^{-1}x)\ .\
\end{align}
We have already shown in Remark \ref{particular} that  $\omega_{R_2,R_1}(x)=-\psi_{R_2,R_1}(x)$ in two particular situations:
rotations that move $y_1=R_1^{-1}x$ to $y_2=R_2^{-1}x$
without turning the tangent plane and rotations that turn the tangent plane without moving the point.
In the next statement, by combining these two particular cases,
we prove that actually they coincide always.
\begin{lemma}\label{lemma angolo}
Let $x\in U_{R_1}\cap U_{R_2}$, then
$\omega_{R_2,R_1}(x)=-\psi_{R_2,R_1}(x)$\ .
\end{lemma}
\begin{proof}
The matrix $R_2^{-1}R_1$ can be decomposed as $R_2^{-1}R_1=EW$ where $W$ is the product of a rotation around an axis that is orthogonal to the plane $[x_0,y_1]$ bringing $y_1$ to a point having the same colatitude as $y_2$ and of a rotation around the north-south axis taking this point to $y_2$. By Remark \ref{particular} we have
$Wg_{y_1}=g_{Wy_1}=g_{y_2}$. $E$ instead is a rotation around an axis passing by $y_2$ itself.

We have then, thanks to \paref{rot} and \paref{al1}
$$
\dlines{
k_{R_2,R_1}(x)=k_{R_1^{-1}R_2,e}(R_1^{-1}x)=k_{W^{-1}E^{-1},e}(y_1)=
g_{EWy_1}^{-1}EWg_{y_1}
=g_{y_2}^{-1}Eg_{y_2}=k_{E^{-1},e}(y_2)\ .\cr
}
$$
By the previous discussion, $\omega_{E^{-1},e}(y_2)=-\psi_{E^{-1},e}(y_2)$.
To finish the proof it is enough to show that
\begin{equation}\label{ss}
\psi_{R_2,R_1}(x)=\psi_{E^{-1},e}(y_2)\ .
\end{equation}
Let us denote $\rho(x)=\rho_e(x)$ the tangent vector at $x$ which is parallel to the curve $\vartheta=const$ and pointing in the direction of increasing $\varphi$. Then in coordinates
$$
\rho(x)=\frac 1{\sqrt{x_1^2+x_2^2}}\ \bigl(-x_2,x_1,0\bigr)
$$
and the action of $R$ is given by (\cite{marinuccigeller},\S3) $\rho_R(x)=R\rho(R^{-1}x)$. As $W\rho(y_1)=\rho(y_2)$ ($W$ does not change the orientation of the tangent plane),
$$
\dlines{
\langle \rho_{R_2}(x), \rho_{R_1}(x) \rangle =
\langle R_2 \rho(R_2^{-1}x), R_1 \rho(R_1^{-1}x) \rangle
= \langle R_1^{-1}R_2 \rho(R_2^{-1}x),\rho(R_1^{-1}x) \rangle=\cr
=\langle W^{-1}E^{-1} \rho(EWR_1^{-1}x), \rho(W^{-1}E^{-1}R_2^{-1}x) \rangle=
\langle E^{-1} \rho(Ey_2), W\rho(W^{-1}y_2) \rangle =\cr
=\langle E^{-1} \rho(y_2)),  W\rho(y_1) \rangle=\langle E^{-1} \rho(y_2)), \rho(y_2) \rangle\ ,
}
$$
so that the oriented angle $\psi_{R_2,R_1}(x)$ between $\rho_{R_2}(x)$ and $\rho_{R_1}(x)$ is actually the rotation angle of $E^{-1}$.

\end{proof}
\section{Appendix}
\begin{prop}\label{real-sq}
Let $\phi$ be a real positive definite function on a compact group $G$,
then there exists a real function $f$ such that $\phi=f*\breve f$.
\end{prop}
\begin{proof}
Let
$$
\phi(g)=\sum_{\sigma\in\widehat G}
\phi^\sigma(g)=\sum_{\sigma\in\widehat
G}\sqrt{\dim
\sigma}\,\tr(\widehat\phi(\sigma)D^\sigma(g))
$$
be the Peter-Weyl decomposition of $\phi$ into isotypical components. We know
that the Hermitian matrices $\widehat\phi(\sigma)$ are positive
definite, so that there exist square roots
$\widehat\phi(\sigma)^{1/2}$ i.e. matrices such that
$\widehat\phi(\sigma)^{1/2}{\widehat\phi(\sigma)^{1/2}}^*=\widehat\phi(\sigma)$
and the functions
$$
f(g)=\sum_{\sigma\in\widehat G}\sqrt{\dim
\sigma}\,\tr(\widehat\phi(\sigma)^{1/2}D^\sigma(g))
$$
are such that $\phi=f*\breve f$. We need to prove that these square
roots can be chosen in such a way that $f$ is also real. Recall that
a representation of a compact group $G$ can be classified as being
of real, complex or quaternionic type (see \cite{B-D}, p. 93 e.g. for details).

\tin{a)} If $\sigma$ is of real type then there exists a conjugation $J$
of $H_\sigma\subset L^2(G)$ such that $J^2=1$. A conjugation is a
$G$-equivariant antilinear endomorphism. It is well known that in
this case one can choose a basis $v_1,\dots, v_{d_\sigma}$ of
$H_\sigma$ formed of ``real'' vectors, i.e. such that $Jv_i=v_i$. It
is then immediate that the representative matrix $D^\sigma$ of the
action of $G$ on $H_\sigma$ is real. Actually, as $J$ is equivariant
and $Jv_i=v_i$,
$$
D^\sigma_{ij}(g)=\langle gv_j,v_i\rangle=\overline{\langle
Jgv_j,Jv_i\rangle}=\overline{\langle
gv_j,v_i\rangle}=\overline{D^\sigma_{ij}(g)}\ .
$$
With this choice of the basis, the matrix $\widehat\phi(\sigma)$ is real
and also $\widehat\phi(\sigma)^{1/2}$ can be chosen to be real and
$g\mapsto\sqrt{\dim \sigma}\,
\tr(\widehat\phi(\sigma)^{1/2}D^\sigma(g))$ turns out to be real
itself.
\tin{b)} If $\sigma$ is of complex type, then
it is not isomorphic to its dual representation $\sigma^*$.
As $D^{\sigma^*}(g):=D^\sigma (g^{-1})^t=
\overline{D^\sigma (g)}$ and $\phi$ is real-valued,  we have
\begin{equation*}\label{easy}
\widehat \phi(\sigma^*) = \overline{\widehat \phi(\sigma)}\ ,
\end{equation*}
so that we can choose $\widehat \phi(\sigma^*)^{1/2} = \overline{\widehat \phi(\sigma)^{1/2}}$ and, as $\sigma$ and $\sigma^*$ have the same dimension, the function
$$
g\mapsto \sqrt{\dim \sigma} \tr(\widehat \phi(\sigma)^{1/2} D^\sigma(g))+ \sqrt{\dim \sigma^*} \tr(\widehat \phi(\sigma^*)^{1/2} D^{\sigma^*}(g))
$$
turns out to be real.
\tin{c)} If $\sigma$ is quaternionic, let $J$ be the corresponding
conjugation. It is immediate that the vectors $v$ and $Jv$ are
orthogonal and from this it follows that
$\dim \sigma=2k$ and that there exists an orthogonal basis for $H_\sigma$
of the form
\begin{equation}\label{anti-basis}
v_1,  \dots, v_k, w_1=J(v_1), \dots, w_k=J(v_k)\ .
\end{equation}
In such a basis the representation matrix of any linear transformation $U:H_\sigma\to H_\sigma$ which commutes with $J$ has the form
\begin{equation}\label{eq blocchi-gen}
\begin{pmatrix}
A  &  B\\
\noalign{\vskip4pt}
-\overline{B} & \overline{A} \\
\end{pmatrix}
\end{equation}
and in particular $D^\sigma(g)$ takes the form
\begin{equation}\label{eq blocchi}
D^\sigma(g)=\begin{pmatrix}
A(g)  &  B(g)\\
\noalign{\vskip4pt}
-\overline{B(g)} & \overline{A(g)} \\
\end{pmatrix}\ .
\end{equation}
By \paref{eq blocchi} we have also, $\phi$ being real valued,
\begin{equation}\label{matr per fi}
\widehat \phi(\sigma) = \begin{pmatrix}
\int_G \phi(g)A(g^{-1})\,dg  &  \int_G \phi(g)B(g^{-1})\,dg\\
\noalign{\vskip6pt}
 -\int_G \phi(g)\overline{B(g^{-1})}\,dg & \int_G
\phi(g)\overline{A(g^{-1})}\,dg \\
\end{pmatrix}:=\begin{pmatrix}
\phi_A   &  \phi_B \\
\noalign{\vskip4pt}
-\overline{ \phi_B} & \overline{\phi_A } \\
\end{pmatrix}\ .
\end{equation}

More interestingly, if $\phi$ is any function such that, with respect to the basis above,
$\widehat\phi(\sigma)$ is of the form \paref{matr per fi}, then the corresponding component
$\phi^\sigma$ is necessarily a real valued function: actually
$$
\dlines{
\phi^\sigma(g)=\tr(\widehat\phi(\sigma)D^\sigma(g))  =\cr
=\tr\bigl(\phi_A A(g)-\phi_B\overline {B(g)}-\overline{\phi_B}B(g)+\overline {\phi_A}\,\overline {A(g)}\bigr)=\tr\bigl(\phi_A A(g)+\overline{\phi_A A(g)}\bigr)-
\tr\bigl(\phi_B\overline {B(g)}+\overline{\phi_B\overline {B(g)}}\bigr)\ .\cr
}
$$
We now prove that the Hermitian square root, $U$ say, of $\widehat\phi(\sigma)$ is of the form
\paref{matr per fi}. Actually note that $\widehat\phi(\sigma)$ is self-adjoint, so that it
can be diagonalized and all its eigenvalues are real (and positive by Proposition
\ref{structure of positive definite} a)). Let $\lambda$ be an eigenvalue and $v$ a
corresponding eigenvector. Then, as
$$
\widehat\phi(\sigma)Jv=J\widehat\phi(\sigma)=J\lambda v=\lambda Jv\ ,
$$
$Jv$ is also an eigenvector associated to $\lambda$. Therefore there exists a basis as in
\paref{anti-basis} that is formed of eigenvectors, i.e. of the form
$v_1,\dots,v_k, w_1,\dots,w_k$ with $Jv_j=w_j$ and
$v_j$ and $w_j$ associated to the same positive eigenvalue $\lambda_j$. In this basis
$\widehat\phi(\sigma)$ is of course diagonal with the (positive) eigenvalues on the diagonal.
Its Hermitian square root $U$ is also diagonal, with the square roots of the eigenvalues
on the diagonal. Therefore $U$ is also the form \paref{matr per fi} and the corresponding
function $\psi(g)=\tr(UD(g))$ is real valued and such that $\psi*\breve\psi=\phi^\sigma$.
\begin{remark}\label{rod}\rm Rodrigues formula for the Legendre polynomials states that
$$
P_\ell(x)=\frac 1{2^\ell\ell!}\,\frac
{d^\ell\hfil}{dx^\ell}\,(x^2-1)^\ell\ .
$$
Therefore
\begin{equation}\label{integral}
\int_0^1P_\ell(x)\, dx=\frac 1{2^\ell\ell!}\,\frac
{d^{\ell-1}\hfil}{dx^{\ell-1}}\,(x^2-1)^\ell\Big|^1_0\ .
\end{equation}
The primitive vanishes at $1$, as the polynomial $(x^2-1)^\ell$ has
a zero of order $\ell$ at $x=1$ and all its derivatives up to the
order $\ell-1$ vanish at $x=1$. In order to compute the primitive at
$0$ we make the binomial expansion of $(x^2-1)^\ell$ and take the
result of the $(\ell-1)$-th derivative of the term of order $\ell-1$
of the expansion. This is actually the term of order $0$ of the
primitive. If $\ell$ is even then $\ell-1$ is odd so that this term
of order $\ell-1$ does not exist (in the expansion only even
powers of $x$ can appear). If $\ell=2m+1$, then the term of order
$\ell-1=2m$ in the expansion is
$$
(-1)^m{2m+1\choose m} z^{2m}
$$
and the result of the integral in \paref{integral} is actually, as
given in \paref{2m},
$$
(-1)^{m+1}\,\frac {(2m)!}{2^{2m+1}(2m+1)!}\,{2m+1\choose m}\ \cdotp
$$
\end{remark}
\end{proof}

\bibliography{bibbase}
\bibliographystyle{amsplain}

\noindent Dipartimento di Matematica, Universit\`a di Roma Tor Vergata, Via della Ricerca Scientifica, 00133 Roma, Italy
\smallskip

{\tt baldi@mat.uniroma2.it}
\smallskip

{\tt rossim@mat.uniroma2.it}
\end{document}